\documentclass[a4paper, british]{amsart}

\usepackage{libertine}
\usepackage[libertine]{newtxmath}

\usepackage{babel}
\usepackage{enumitem}
\usepackage{hyperref}
\usepackage[utf8]{inputenc}
\usepackage{newunicodechar}
\usepackage{mathtools}
\usepackage{varioref}
\usepackage[arrow,curve,matrix]{xy}

\usepackage{colortbl}
\usepackage{graphicx}
\usepackage{tikz}

\usepackage{mathrsfs}

\definecolor{linkred}{rgb}{0.7,0.2,0.2}
\definecolor{linkblue}{rgb}{0,0.2,0.6}

\numberwithin{figure}{section}

\usepackage[hyperpageref]{backref}

\setdescription{labelindent=\parindent, leftmargin=2\parindent}
\setitemize[1]{labelindent=\parindent, leftmargin=2\parindent}
\setenumerate[1]{labelindent=0cm, leftmargin=*, widest=iiii}

\newunicodechar{א}{\ensuremath{\aleph}}
\newunicodechar{α}{\ensuremath{\alpha}}
\newunicodechar{β}{\ensuremath{\beta}}
\newunicodechar{χ}{\ensuremath{\chi}}
\newunicodechar{δ}{\ensuremath{\delta}}
\newunicodechar{ε}{\ensuremath{\varepsilon}}
\newunicodechar{Δ}{\ensuremath{\Delta}}
\newunicodechar{η}{\ensuremath{\eta}}
\newunicodechar{γ}{\ensuremath{\gamma}}
\newunicodechar{Γ}{\ensuremath{\Gamma}}
\newunicodechar{ι}{\ensuremath{\iota}}
\newunicodechar{κ}{\ensuremath{\kappa}}
\newunicodechar{λ}{\ensuremath{\lambda}}
\newunicodechar{Λ}{\ensuremath{\Lambda}}
\newunicodechar{ν}{\ensuremath{\nu}}
\newunicodechar{μ}{\ensuremath{\mu}}
\newunicodechar{ω}{\ensuremath{\omega}}
\newunicodechar{Ω}{\ensuremath{\Omega}}
\newunicodechar{π}{\ensuremath{\pi}}
\newunicodechar{Π}{\ensuremath{\Pi}}
\newunicodechar{φ}{\ensuremath{\phi}}
\newunicodechar{Φ}{\ensuremath{\Phi}}
\newunicodechar{ψ}{\ensuremath{\psi}}
\newunicodechar{Ψ}{\ensuremath{\Psi}}
\newunicodechar{ρ}{\ensuremath{\rho}}
\newunicodechar{σ}{\ensuremath{\sigma}}
\newunicodechar{Σ}{\ensuremath{\Sigma}}
\newunicodechar{τ}{\ensuremath{\tau}}
\newunicodechar{θ}{\ensuremath{\theta}}
\newunicodechar{Θ}{\ensuremath{\Theta}}
\newunicodechar{ξ}{\ensuremath{\xi}}
\newunicodechar{Ξ}{\ensuremath{\Xi}}
\newunicodechar{ζ}{\ensuremath{\zeta}}

\newunicodechar{ℓ}{\ensuremath{\ell}}
\newunicodechar{ï}{\"{\i}}

\newunicodechar{𝔸}{\ensuremath{\bA}}
\newunicodechar{𝔹}{\ensuremath{\bB}}
\newunicodechar{ℂ}{\ensuremath{\bC}}
\newunicodechar{𝔻}{\ensuremath{\bD}}
\newunicodechar{𝔼}{\ensuremath{\bE}}
\newunicodechar{𝔽}{\ensuremath{\bF}}
\newunicodechar{𝔾}{\ensuremath{\bG}}
\newunicodechar{ℕ}{\ensuremath{\bN}}
\newunicodechar{ℙ}{\ensuremath{\bP}}
\newunicodechar{ℚ}{\ensuremath{\bQ}}
\newunicodechar{ℝ}{\ensuremath{\bR}}
\newunicodechar{𝕏}{\ensuremath{\bX}}
\newunicodechar{ℤ}{\ensuremath{\bZ}}
\newunicodechar{𝒜}{\ensuremath{\sA}}
\newunicodechar{ℬ}{\ensuremath{\sB}}
\newunicodechar{𝒞}{\ensuremath{\sC}}
\newunicodechar{𝒟}{\ensuremath{\sD}}
\newunicodechar{ℰ}{\ensuremath{\sE}}
\newunicodechar{ℱ}{\ensuremath{\sF}}
\newunicodechar{𝒢}{\ensuremath{\sG}}
\newunicodechar{ℋ}{\ensuremath{\sH}}
\newunicodechar{𝒥}{\ensuremath{\sJ}}
\newunicodechar{ℒ}{\ensuremath{\sL}}
\newunicodechar{ℳ}{\ensuremath{\sM}}
\newunicodechar{𝒪}{\ensuremath{\sO}}
\newunicodechar{𝒬}{\ensuremath{\sQ}}
\newunicodechar{𝒮}{\ensuremath{\sS}}
\newunicodechar{𝒯}{\ensuremath{\sT}}
\newunicodechar{𝒲}{\ensuremath{\sW}}

\newunicodechar{∂}{\ensuremath{\partial}}
\newunicodechar{∇}{\ensuremath{\nabla}}

\newunicodechar{↺}{\ensuremath{\circlearrowleft}}
\newunicodechar{∞}{\ensuremath{\infty}}
\newunicodechar{⊕}{\ensuremath{\oplus}}
\newunicodechar{⊗}{\ensuremath{\otimes}}
\newunicodechar{•}{\ensuremath{\bullet}}
\newunicodechar{Λ}{\ensuremath{\wedge}}
\newunicodechar{↪}{\ensuremath{\into}}
\newunicodechar{→}{\ensuremath{\to}}
\newunicodechar{↦}{\ensuremath{\mapsto}}
\newunicodechar{⨯}{\ensuremath{\times}}
\newunicodechar{∪}{\ensuremath{\cup}}
\newunicodechar{∩}{\ensuremath{\cap}}
\newunicodechar{⊋}{\ensuremath{\supsetneq}}
\newunicodechar{⊇}{\ensuremath{\supseteq}}
\newunicodechar{⊃}{\ensuremath{\supset}}
\newunicodechar{⊊}{\ensuremath{\subsetneq}}
\newunicodechar{⊆}{\ensuremath{\subseteq}}
\newunicodechar{⊂}{\ensuremath{\subset}}
\newunicodechar{⊄}{\ensuremath{\not \subset}}
\newunicodechar{≥}{\ensuremath{\geq}}
\newunicodechar{≠}{\ensuremath{\neq}}
\newunicodechar{≫}{\ensuremath{\gg}}
\newunicodechar{≪}{\ensuremath{\ll}}

\newunicodechar{≤}{\ensuremath{\leq}}
\newunicodechar{∈}{\ensuremath{\in}}
\newunicodechar{∉}{\ensuremath{\not \in}}
\newunicodechar{∖}{\ensuremath{\setminus}}
\newunicodechar{◦}{\ensuremath{\circ}}
\newunicodechar{°}{\ensuremath{^\circ}}
\newunicodechar{…}{\ifmmode\mathellipsis\else\textellipsis\fi}
\newunicodechar{·}{\ensuremath{\cdot}}
\newunicodechar{⋯}{\ensuremath{\cdots}}
\newunicodechar{∅}{\ensuremath{\emptyset}}
\newunicodechar{⇒}{\ensuremath{\Rightarrow}}

\newunicodechar{⁰}{\ensuremath{^0}}
\newunicodechar{¹}{\ensuremath{^1}}
\newunicodechar{²}{\ensuremath{^2}}
\newunicodechar{³}{\ensuremath{^3}}
\newunicodechar{⁴}{\ensuremath{^4}}
\newunicodechar{⁵}{\ensuremath{^5}}
\newunicodechar{⁶}{\ensuremath{^6}}
\newunicodechar{⁷}{\ensuremath{^7}}
\newunicodechar{⁸}{\ensuremath{^8}}
\newunicodechar{⁹}{\ensuremath{^9}}
\newunicodechar{ⁱ}{\ensuremath{^i}}
\newunicodechar{⁺}{\ensuremath{^+}}

\newunicodechar{⌈}{\ensuremath{\lceil}}
\newunicodechar{⌉}{\ensuremath{\rceil}}
\newunicodechar{⌊}{\ensuremath{\lfloor}}
\newunicodechar{⌋}{\ensuremath{\rfloor}}

\newunicodechar{≅}{\ensuremath{\cong}}
\newunicodechar{⇔}{\ensuremath{\Leftrightarrow}}
\newunicodechar{∃}{\ensuremath{\exists}}
\newunicodechar{±}{\ensuremath{\pm}}

\DeclareFontFamily{OMS}{rsfs}{\skewchar\font'60}
\DeclareFontShape{OMS}{rsfs}{m}{n}{<-5>rsfs5 <5-7>rsfs7 <7->rsfs10 }{}
\DeclareSymbolFont{rsfs}{OMS}{rsfs}{m}{n}
\DeclareSymbolFontAlphabet{\scr}{rsfs}
\DeclareSymbolFontAlphabet{\scr}{rsfs}

\DeclareFontFamily{U}{mathx}{\hyphenchar\font45}
\DeclareFontShape{U}{mathx}{m}{n}{
      <5> <6> <7> <8> <9> <10>
      <10.95> <12> <14.4> <17.28> <20.74> <24.88>
      mathx10
      }{}
\DeclareSymbolFont{mathx}{U}{mathx}{m}{n}
\DeclareFontSubstitution{U}{mathx}{m}{n}
\DeclareMathAccent{\wcheck}{0}{mathx}{"71}

\DeclareMathOperator{\codim}{codim}

\DeclareMathOperator{\Hom}{Hom}
\DeclareMathOperator{\Id}{Id}

\DeclareMathOperator{\img}{img}
\DeclareMathOperator{\Pic}{Pic}
\DeclareMathOperator{\rank}{rank}

\DeclareMathOperator{\red}{red}
\DeclareMathOperator{\reg}{reg}

\DeclareMathOperator{\sing}{sing}

\DeclareMathOperator{\Sym}{Sym}
\DeclareMathOperator{\supp}{supp}
\DeclareMathOperator{\tor}{tor}

\newcommand{\sA}{\scr{A}}
\newcommand{\sB}{\scr{B}}
\newcommand{\sC}{\scr{C}}
\newcommand{\sD}{\scr{D}}
\newcommand{\sE}{\scr{E}}
\newcommand{\sF}{\scr{F}}
\newcommand{\sG}{\scr{G}}
\newcommand{\sH}{\scr{H}}
\newcommand{\sHom}{\scr{H}\negthinspace om}

\newcommand{\sJ}{\scr{J}}

\newcommand{\sL}{\scr{L}}
\newcommand{\sM}{\scr{M}}

\newcommand{\sO}{\scr{O}}

\newcommand{\sQ}{\scr{Q}}

\newcommand{\sS}{\scr{S}}
\newcommand{\sT}{\scr{T}}

\newcommand{\sV}{\scr{V}}
\newcommand{\sW}{\scr{W}}

\newcommand{\cC}{\mathcal C}

\newcommand{\bA}{\mathbb{A}}
\newcommand{\bB}{\mathbb{B}}
\newcommand{\bC}{\mathbb{C}}
\newcommand{\bD}{\mathbb{D}}
\newcommand{\bE}{\mathbb{E}}
\newcommand{\bF}{\mathbb{F}}
\newcommand{\bG}{\mathbb{G}}

\newcommand{\bN}{\mathbb{N}}

\newcommand{\bP}{\mathbb{P}}
\newcommand{\bQ}{\mathbb{Q}}
\newcommand{\bR}{\mathbb{R}}

\newcommand{\bX}{\mathbb{X}}

\newcommand{\bZ}{\mathbb{Z}}

\theoremstyle{plain}
\newtheorem{thm}{Theorem}[section]

\newtheorem{cor}[thm]{Corollary}
\newtheorem{defn}[thm]{Definition}

\newtheorem{lem}[thm]{Lemma}

\newtheorem{prop}[thm]{Proposition}

\theoremstyle{remark}
\newtheorem{assumption}[thm]{Assumption}
\newtheorem{asswlog}[thm]{Assumption w.l.o.g.}
\newtheorem{claim}[thm]{Claim}
\newtheorem{c-n-d}[thm]{Claim and Definition}

\newtheorem{construction}[thm]{Construction}

\newtheorem{explanation}[thm]{Explanation}
\newtheorem{notation}[thm]{Notation}

\newtheorem{rem}[thm]{Remark}

\newtheorem*{rem-nonumber}{Remark}
\newtheorem{setting}[thm]{Setting}

\numberwithin{equation}{thm}

\setlist[enumerate]{label=(\thethm.\arabic*), before={\setcounter{enumi}{\value{equation}}}, after={\setcounter{equation}{\value{enumi}}}}

\newcommand{\into}{\hookrightarrow}

\newcommand{\wtilde}{\widetilde}
\newcommand{\what}{\widehat}

\newcommand\CounterStep{\addtocounter{thm}{1}\setcounter{equation}{0}}

\newcommand{\factor}[2]{\left. \raise 2pt\hbox{$#1$} \right/\hskip -2pt\raise -2pt\hbox{$#2$}}
\newcommand{\Preprint}[1]{}

\newcommand{\subversionInfo}{}
\newcommand{\svnid}[1]{}
\newcommand{\approvals}[2][Approval]{}

\renewcommand{\phi}{\varphi}

\usepackage{imakeidx}
\makeindex[title=Index of Concepts]
\indexsetup{firstpagestyle=headings,headers={STEFAN KEBEKUS, ERWAN ROUSSEAU AND FRÉDÉRIC TOUZET}{$\cC$-PAIRS AND THEIR MORPHISMS}}

\usepackage{cancel}
\usepackage{tikz-cd}
\tikzset{commutative diagrams/arrow style=tikz}
\author{Stefan Kebekus}
\address{Stefan Kebekus, Mathematisches Institut, Albert-Ludwigs-Universität Freiburg, Ernst-Zermelo-Straße 1, 79104 Freiburg im Breisgau, Germany}
\email{\href{mailto:stefan.kebekus@math.uni-freiburg.de}{stefan.kebekus@math.uni-freiburg.de}}
\urladdr{\url{https://cplx.vm.uni-freiburg.de}}

\author{Erwan Rousseau}
\address{Erwan Rousseau, Univ Brest, CNRS UMR 6205,	Laboratoire de Mathematiques de Bretagne Atlantique\\ F-29200 Brest, France}
\email{\href{mailto:erwan.rousseau@univ-brest.fr}{erwan.rousseau@univ-brest.fr}}
\urladdr{\href{http://eroussea.perso.math.cnrs.fr/}{http://eroussea.perso.math.cnrs.fr}}

\author{Frédéric Touzet}
\address{Frédéric Touzet, Univ Rennes, CNRS, IRMAR - UMR 6625, 35000, Rennes, France}
\email{\href{mailto:frederic.touzet@univ-rennes.fr}{frederic.touzet@univ-rennes.fr}}
\urladdr{\href{https://irmar.univ-rennes.fr/en/node/273}{https://irmar.univ-rennes.fr/en/node/273}}



\title{Irregularities of special $\cC$-pairs}
\date{\today}

\makeatletter
\hypersetup{
  pdfauthor={\authors},
  pdftitle={\@title},
  pdfsubject={\@subjclass},
  pdfkeywords={\@keywords},
  pdfstartview={Fit},
  pdfpagelayout={TwoColumnRight},
  pdfpagemode={UseOutlines},
  bookmarks,
  colorlinks,
  linkcolor=linkblue,
  citecolor=linkred,
  urlcolor=linkred
}
\makeatother

\DeclareMathOperator{\Alb}{Alb}

\DeclareMathOperator{\Div}{Div}

\DeclareMathOperator{\orb}{orb}

\DeclareMathOperator{\res}{res}
\DeclareMathOperator{\snc}{snc}

\theoremstyle{plain}

\theoremstyle{remark}

\newtheorem{conj}[thm]{Conjecture}

\newtheorem{reminder}[thm]{Reminder}

\allowdisplaybreaks

\begin{document}

\approvals[Approval for Abstract]{Erwan & yes \\ Frédéric & yes \\ Stefan & yes}
\begin{abstract}
\selectlanguage{british}

This paper studies irregularity-type invariants of special $\cC$-pairs, or
``geometric orbifolds'' in the sense of Campana.  Under mild assumptions on the
singularities, we show that the augmented irregularity of a $\cC$-pair $(X,D)$
is bounded by its dimension.  This generalizes earlier results of Campana, and
strengthens known results even in the classic case where $X$ is a projective
manifold and $D = 0$.  The proof builds on new extension results for adapted
forms, analysis of foliations on Albanese varieties, and constructions of
Bogomolov sheaves using strict wedge subspaces of adapted forms.

\end{abstract}

\maketitle
\tableofcontents

%
%
\svnid{$Id: 01-intro.tex 138 2026-01-12 07:18:20Z touzet $}
\selectlanguage{british}

\section{Introduction}
\subversionInfo
\approvals{Erwan & yes \\ Frédéric & yes \\ Stefan & yes}

Special varieties were introduced in a series of influential papers by Campana,
\cite{Cam04, MR2831280}, as complex-projective manifolds where the classic
Bogomolov-Sommese inequality is strict, or equivalently, as complex-projective
manifolds that do not dominate a ``geometric orbifold'' or ``$\cC$-pair'' of
general type.  It is conjectured that the notion of ``specialness'' characterizes
``potential density'', both in the arithmetic setting (where ``density'' refers
to sets of rational points) and in the analytic setting (where ``density''
refers to entire curves).

This article studies \emph{irregularities} of special manifolds and of mildly
singular $\cC$-pairs that appear in the minimal model program.

\subsubsection*{Invariants of special manifolds}
\approvals{Erwan & yes \\ Frédéric & yes \\ Stefan & yes}

The starting point is a fundamental observation of Campana: If $X$ is a
complex-projective manifold that is special, then the irregularity $q(X):=h⁰(X,
Ω_X¹)$ is always bounded by the dimension of $X$, \cite[Sect.~5.2]{Cam04}.  Using
his results on the invariance of specialness under étale coverings, he concludes
that the \emph{augmented irregularity},
\[
  \wtilde{q}(X) :=\sup \left\{ q\bigl(\wtilde{X}\bigr) \::\: \wtilde X → X \text{ a finite étale cover} \right\},
\]
is likewise bounded by $\dim X$.

\subsubsection*{Invariants of special pairs}
\approvals{Erwan & yes \\ Frédéric & yes \\ Stefan & yes}

Given that the natural objects of Campana's theory are ``geometric orbifolds''
or ``$\cC$-pairs'', where adapted differentials take the role that ordinary
differentials play for ordinary spaces, it is natural to ask for
generalizations.

\begin{conj}[Irregularities of special $\cC$-pairs, \protect{\cite[Conjecture~6.17]{orbiAlb1}}]\label{conj:1-1}%
  Let $(X, D)$ be a $\cC$-pair where the analytic variety $X$ is compact and
  Kähler.  If $(X, D)$ is special, then its augmented irregularity is bounded by
  the dimension, $q⁺(X, D) ≤ \dim X$.
\end{conj}

For the convenience of the reader not familiar with the theory, we recall the
definition of ``augmented irregularity'' in brief.  The reference paper
\cite{orbiAlb1} introduces and discusses all relevant notions in great detail.

\begin{defn}[Irregularity, augmented irregularity, \cite{orbiAlb1}]
  Let $(X, D)$ be a compact $\cC$-pair.  If $γ : \what{X} → X$ is any cover, we
  refer to the number
  \begin{align*}
    q(X, D, γ) & := h⁰ \left(\what{X}, Ω^{[1]}_{(X,D,γ)} \right) \\
    \intertext{as the \emph{irregularity} of $(X, D, γ)$.  The number}
    q⁺(X, D) & := \sup \bigl\{q(X, D, γ) \::\: γ \text{ a cover} \bigr\} ∈ ℕ ∪ \{∞\}
  \end{align*}
  is the \emph{augmented irregularity} of the $\cC$-pair $(X, D)$.
\end{defn}

\subsection{Main result}
\approvals{Erwan & yes \\ Frédéric & yes \\ Stefan & yes}

The main result of this paper answers Conjecture~\ref{conj:1-1} in the positive,
for all pairs that will typically appear in minimal model theory.  We refer the
reader to \cite{KM98} for the definition of ``divisorially log terminal'' pairs.

\begin{thm}[Boundedness of augmented irregularity]\label{thm:1-3}%
  Let $(X, D)$ be a $\cC$-pair that satisfies one of the following conditions.
  \begin{enumerate}
  \item\label{il:1-3-1} The analytic variety $X$ is compact Kähler and $(X,D)$
    is locally uniformizable.

  \item\label{il:1-3-2} The analytic variety $X$ is projective and $(X,D)$ is
    divisorially log terminal (=dlt).
  \end{enumerate}
  If $(X,D)$ is special, then $q⁺(X, D) ≤ \dim X$.
\end{thm}

\begin{rem}[Novelty of the result]
  Even in cases where $X$ is a complex-projective manifold and $D = 0$,
  Theorem~\ref{thm:1-3} is new and stronger than Campana's classic result, which
  considers étale coverings only.
\end{rem}

\begin{rem}[Earlier results on Albanese irregularities]
  Theorem~\ref{thm:1-3} generalizes and strengthens earlier results, including
  \cite[Theorem~8.1]{orbialb2}, on the ``augmented Albanese irregularity''
  $q⁺_{\Alb}(X, D)$.  The augmented Albanese irregularity is a variant of the
  augmented irregularity.  It is geometrically meaningful, always bounded by
  $q⁺(X, D)$, but hard to control and compute in practise.
\end{rem}

The proof of Theorem~\ref{thm:1-3} relies in part on the following extension
theorem for adapted reflexive differentials, which might be of independent
interest.

\begin{thm}[Extension of adapted forms on dlt pairs]\label{thm:1-6}%
  Let $(X, D)$ be an algebraic, quasi-projective $\cC$-pair that is dlt.  Then,
  adapted reflexive 1-forms on $(X,D)$ extend to log resolutions of
  singularities.
\end{thm}

\subsection{Outline of the proof}
\approvals{Erwan & yes \\ Frédéric & yes \\ Stefan & yes}

For the proof of the classic result on irregularities of special
complex-projective manifolds, it suffices to show that the Albanese map of a
special manifold is necessarily surjective.  Aiming for the contrapositive,
Campana studies manifolds whose Albanese is \emph{not} surjective.  Building on
earlier of work of Kawamata, Kobayashi, and Ueno on positivity in sheaves of
differentials on varieties of maximal Albanese dimension, he constructs a sheaf
of differentials where equality in the Bogomolov-Sommese inequality is attained,
showing that the underlying space cannot be special.

In our setting, where adapted differentials take the role that ordinary
differentials play for ordinary spaces, there is in general no ``adapted
Albanese map''.  Even in special cases where a suitably-defined Albanese does
exist, it is known that the classical equality between the irregularity and the
dimension of the Albanese variety is not true in general
\cite[Sect.~7.1]{orbialb2}.

We overcome this problem by considering the classic Albanese of suitable covers,
where adapted differentials define a foliation.  Though these will typically have
Zariski dense leaves, we can leverage ideas from Catanese's work on generalized
Castelnuovo-De Franchis theorems and ``strict wedge subspaces of
differentials'', in order to obtain positivity results for the foliated variety
that can be used in lieu of the classic arguments.

For sheaves of one-forms, these arguments work particularly well, and provide
the following partial generalization of Campana's statement on the invariance of
specialness under étale cover, \cite[Sect.~5.2]{Cam04} and
\cite[Prop.~10.11]{MR2831280}.

\begin{cor}[Adapted one-forms on covers of spacial pairs]\label{cor:1-7}%
  Let $(X, D)$ be a projective $\cC$-pair that is dlt, let $γ : \what{X}
  \twoheadrightarrow X$ be any cover and let $ℒ ⊆ Ω^{[1]}_{(X,D,γ)}$ be coherent
  of rank one.  If $(X,D)$ is special, then the $\cC$-Kodaira-Iitaka dimension
  of $ℒ$ is bounded by one: $κ_{\cC}(ℒ) < 1$.
\end{cor}

\subsection{Acknowledgements}
\approvals{Erwan & yes \\ Frédéric & yes \\ Stefan & yes}

The authors would like to thank Finn Bartsch, Frédéric Campana and Ariyan
Javanpeykar for long and fruitful discussions.

The work on this paper was carried out in part while Stefan Kebekus visited the
\foreignlanguage{french}{Université de Bretagne Occidentale} at Brest.  He would
like to thank the department for its hospitality and the pleasant working
atmosphere.

%
%
\svnid{$Id: 02-notation.tex 118 2025-12-15 14:26:08Z rousseau $}
\selectlanguage{british}

\section{Notation and known results}
\subversionInfo
\approvals{Erwan & yes \\ Frédéric & yes \\ Stefan & yes}

This paper works with complex spaces.  With very few exceptions, we follow the
notation of the standard reference texts \cite{CAS, DemaillyBook}.  This section
clarifies less commonly-used notation and recalls a few well-known results for
later reference.  A full introduction to the theory of $\cC$-pairs is, however,
out of scope.  We refer the reader to the reference \cite{orbiAlb1} for
definitions and a very detailed introduction to all the material used here.  For
the reader's convenience, we include precise references to \cite{orbiAlb1}
throughout the text, whenever a term of $\cC$-pair theory appears for the first
time.

\subsection{Global assumptions and standard notation}
\approvals{Erwan & yes \\ Frédéric & yes \\ Stefan & yes}

An \emph{analytic variety} is a reduced, irreducible complex space.  For
clarity, we refer to holomorphic maps between analytic varieties as
\emph{morphisms} and reserve the word \emph{map} for meromorphic mappings.

\begin{defn}[Big and small sets]
  Let $X$ be an analytic variety.  An analytic subset $A ⊊ X$ is called
  \emph{small} if it has codimension two or more.  An open set $U ⊆ X$ is called
  \emph{big} if $X∖U$ is analytic and small.
\end{defn}

\begin{defn}[$q$-morphisms]
  Quasi-finite morphisms between normal analytic varieties of equal dimension
  are called \emph{$q$-morphisms}.
\end{defn}

Following the literature, we use square brackets to denote reflexive tensor
operations.

\begin{notation}[Reflexive tensor operations]
  Let $X$ be a normal analytic variety and let $ℒ$ be a torsion free coherent
  sheaf of $𝒪_X$-modules.  Write
  \begin{align*}
    ℒ^{[⊗ n]} & := (ℒ^{⊗ n})^{**}, & Λ^{[n]} ℒ & := (Λ^n ℒ)^{**}, \\
    \Sym^{[n]} ℒ & := (\Sym^n ℒ)^{**}, & \det ℒ & := (Λ^{\rank ℒ} ℒ)^{**}.
  \end{align*}
  If $φ : X → Y$ is a morphism and $ℱ$ a coherent sheaf on $Y$, write $φ^{[*]} ℱ
  := (φ^* ℱ)^{**}$.
\end{notation}

If $X$ is any analytic variety, we denote the sheaf of Kähler differentials by
$Ω^•_X$.  We recall the notation for reflexive logarithmic differentials.

\begin{notation}[NC locus]
  Let $X$ be a normal analytic variety and let $D$ be a Weil $ℚ$-divisor on $X$.
  Write $(X,D)_{\reg} ⊆ X$ for the maximal open set where $X$ is smooth and $D$
  has normal crossing support.
\end{notation}

\begin{notation}[Differentials with logarithmic poles]
  Let $X$ be a normal analytic variety and let $D$ be a Weil $ℚ$-divisor on $X$.
  \begin{enumerate}
  \item If $X$ is smooth and $D$ has nc support, we will often write $Ω^p_X(\log
    D)$ to denote the sheaves of Kähler differentials with logarithmic poles
    along $\supp D$.

  \item\label{il:2-5-2} Denote the inclusion of the nc locus by $ι :
    (X,D)_{\reg} ↪ X$ and write
    \[
      Ω^{[p]}_X(\log D) = ι_* Ω^p_{(X,D)_{\reg}} \bigl(\log D|_{(X,D)_{\reg}} \bigr).
    \]

  \item If $π : Y → X$ is any morphism from a smooth variety $Y$ where $π^{-1}
    \supp D$ is of pure codimension one and has normal crossings, write
    $Ω^p_Y(\log π^* D)$ to denote the sheaves of Kähler differentials with
    logarithmic poles along the set $π^{-1}\supp D$.
  \end{enumerate}
\end{notation}

\begin{rem}
  The sheaf $Ω^{[p]}_X(\log D)$ in \ref{il:2-5-2} is reflexive, in particular
  coherent.
\end{rem}

For later reference, we note the following lemma on the behaviour of pull-back
for differentials under product maps.  The proof is elementary and left to the
reader.

\begin{lem}[Pull-back of differentials under product maps]\label{lem:2-7}%
  Let $(φ_i : X \dashrightarrow Y_i)_{i = 1, …, a}$ be a finite number of
  rational maps between complex manifolds and let
  \[
    φ := φ_1 ⨯ ⋯ ⨯ φ_a : X \dashrightarrow Y_1 ⨯ ⋯ ⨯ Y_a
  \]
  be the associated product map.  Then, there exists a dense open subset $X° ⊆
  X$ where all maps are well-defined and where the following subsheaves of
  $Ω¹_{X°}$ agree,
  \[
    \img d (φ|_{X°}) = \sum_{i = 1}^a \img d (φ_i|_{X°}) ⊆ Ω¹_{X°}.  \eqno \qed
  \]
\end{lem}

\subsection{\texorpdfstring{$\cC$}{C}-pairs and adapted morphisms}
\approvals{Erwan & yes \\ Frédéric & yes \\ Stefan & yes}

The key notion of the present paper is the $\cC$-pair.  We recall the definition
in brief and refer to the reference paper \cite{orbiAlb1} for details.

\begin{defn}[\protect{$\cC$-pairs, \cite[Sect.~2.5]{orbiAlb1}}]
  A $\cC$-pair is a pair $(X, D)$ where $X$ is a normal analytic variety and $D$
  a Weil $ℚ$-divisor $D$ is of the form
  \[
    D = \sum_i \frac{m_i-1}{m_i}·D_i,
  \]
  with $m_i ∈ ℕ^{≥ 2} ∪ \{∞\}$ and $\frac{∞-1}{∞} = 1$.  If $(X, D)$ is a
  $\cC$-pair, it will sometimes be convenient to consider the following Weil
  $ℚ$-divisor
  \[
    D_{\orb} := \sum_{i \:|\: m_i < ∞} \frac{1}{m_i}·D_i ∈ ℚ\Div(X).
  \]
\end{defn}

\begin{defn}[\protect{Adapted morphism, \cite[Sect.~2.5]{orbiAlb1}}]
  Given a $\cC$-pair $(X, D)$, a $q$-morphism $γ : \what{X} → X$ is called
  \emph{adapted for $(X, D)$} if $γ^* D_{\orb}$ is integral.  It is called
  \emph{strongly adapted for $(X, D)$} if $γ^* D_{\orb}$ is reduced.
\end{defn}

\subsection{Linear systems in reflexive sheaves}
\label{sec:2-3}
\approvals{Erwan & yes \\ Frédéric & yes \\ Stefan & yes}

If $X$ is a compact manifold, $ℒ ∈ \Pic(X)$ a line bundle and $L ⊆ H⁰ \bigl(
X,\, ℒ \bigr)$ a non-trivial space of sections, complex geometry frequently
considers the meromorphic map $φ_{L,ℒ} : X \dasharrow ℙ (L)$, given at general
points by $x ↦ \ker \bigl( σ ↦ σ(x) \bigr)$.  Parts of this paper discuss
analogous constructions in cases where $X$ is potentially singular and $ℒ$ is
reflexive of rank one.  To avoid any confusion, we clarify the setup in detail.

\begin{notation}[Projectivized linear spaces]\label{not:2-10}%
  If $L$ is a linear space, write $ℙ(L)$ for the space of hyperplanes in $L$, or
  equivalently, for the space of one-dimensional subspaces in $L^*$.
\end{notation}

\begin{rem}[Subspaces and projections]\label{rem:2-11}%
  With Notation~\ref{not:2-10}, an inclusion $L_1 ↪ L_2$ of linear spaces
  therefore corresponds a linear projection $ℙ(L_2) \dasharrow ℙ(L_1)$, given at
  general points by $H ↦ H ∩ L_1$.
\end{rem}

\begin{lem}[Spaces of sections in torsion free sheaves]\label{lem:2-12}%
  If $X$ is a compact, normal analytic variety, $ℒ$ a torsion free, rank-one
  sheaf of $𝒪_X$-modules and $L ⊆ H⁰ \bigl( X,\, ℒ \bigr)$ a non-trivial space
  of sections, there exists a meromorphic map $φ_{L,ℒ} : X \dasharrow ℙ (L)$,
  given at general points by $x ↦ \ker \bigl( σ ↦ σ(x) \bigr)$.
\end{lem}
\begin{proof}
  Recall from \cite[Thm.~3.5]{Rossi68} that there exists a bimeromorphic
  modification, say $π : \wtilde{X} → X$, where $ℱ := (π^* ℒ) / \tor$ is
  invertible.  Let $F ⊆ H⁰ \bigl( \wtilde{X},\, ℱ \bigr)$ be the image of $L$
  under the natural inclusion
  \[
    π^* / \tor : H⁰ \bigl( X,\, ℒ \bigr) ↪ H⁰ \bigl( \wtilde{X},\, ℱ \bigr),
  \]
  set $φ_{L, ℒ} := φ_{F, ℱ} ◦ π^{-1}$ and observe that the map does not depend
  on the choice of the modification.
\end{proof}

\begin{rem}
  The reader coming from algebraic geometry might wonder why
  Lemma~\ref{lem:2-12} requires any proof at all: in contrast to the algebraic
  setting, holomorphic morphisms defined on a Zariski dense open do not in
  general extend to meromorphic mapping on the full space.
\end{rem}

\begin{notation}[Meromorphic maps induced by linear systems in torsion free sheaves]
  Throughout the present paper, we use the notation $φ_{L,ℒ}$ to denote the
  meromorphic map of Lemma~\ref{lem:2-12}.  In cases where $L = H⁰\bigl( X,\, ℒ
  \bigr)$, we write $φ_ℒ$ instead of $φ_{H⁰\bigl( X,\, ℒ \bigr), ℒ}$.
\end{notation}

\subsection{SNC morphisms}
\approvals{Erwan & yes \\ Frédéric & yes \\ Stefan & yes}

While snc pairs are the logarithmic analogues of smooth spaces, snc morphisms
are the analogues of smooth maps.  We recall the main properties for the
reader's convenience and refer to \cite[Sect.~2.B]{GKKP11} for a full discussion
and for references to the literature.  While \cite{GKKP11} works with algebraic
varieties, the results mentioned here carry over to the analytic setting without
change.

\begin{notation}[Intersection of boundary components]\label{not:2-15}%
  Let $(X,D)$ be a pair,
  \[
    D = α_1·D_1 + … + α_n·D_n.
  \]
  If $I ⊆ \{1, …, n\}$ is not empty, write $D_I := ∩_{i ∈ I} D_i$ for the
  (potentially non-reduced) intersection of complex spaces.  If $I$ is empty,
  set $D_I := X$.
\end{notation}

\begin{reminder}[{Description of snc pairs, \cite[Rem.~2.8]{GKKP11}}]%
  In the setup of Notation~\ref{not:2-15}, the pair $(X,D)$ is snc if and only
  if the following holds for every index set $I ⊆ \{1, …, n\}$ with $D_I \not =
  ∅$.
  \begin{enumerate}
  \item The intersection $D_I$ is smooth.

  \item The codimension equals $\codim_XD_I=|I|$.
  \end{enumerate}
\end{reminder}

\begin{defn}[{Snc morphism, \cite[Def.~2.9]{GKKP11}}]%
  Let $(X,D)$ be an snc pair and let $φ : X \twoheadrightarrow T$ be a
  surjective morphism to a complex manifold.  Call $φ$ \emph{an snc morphism of}
  $(X,D)$ if the following holds for every index set $I ⊆ \{1, …, n\}$ with $D_I
  \not = ∅$.
  \begin{enumerate}
  \item The restricted morphism $φ|_{D_I} : D_I → T$ is smooth.

  \item The restricted morphism $φ|_{D_I}$ has relative dimension $\dim X-\dim T
  -|I|$.
  \end{enumerate}
\end{defn}

\begin{reminder}[All morphisms are generically snc]%
  If $(X,D)$ is an snc pair and $φ : X → T$ is a surjective morphism to a
  complex manifold, then there exists a dense, Zariski open subset of $T$ over
  which $φ$ is an snc morphism.
\end{reminder}

\begin{reminder}[Fibre bundle structure of proper morphisms]\label{remi:2-19}%
  Let $(X,D)$ be a logarithmic snc pair and let $φ : X \twoheadrightarrow T$ be
  a proper snc morphism of $(X,D)$.  Then, $φ : X ∖ D \twoheadrightarrow T$ is a
  differentiable fibre bundle.
\end{reminder}

\begin{reminder}[Fibers and relative differentials of snc morphisms]\label{rem:2-20}%
  Let $(X,D)$ be a logarithmic snc pair and let $φ : X \twoheadrightarrow T$ be
  an snc morphism of $(X,D)$.  Then, there exists a natural exact sequence of
  locally free sheaves,
  \[
    \begin{tikzcd}
      0 \ar[r] & φ^* Ω¹_T \ar[r, hook, "dφ"] & Ω¹_X(\log D) \ar[r, two heads, "q"] & \underbrace{\factor{Ω¹_X(\log D)}{φ^* Ω¹_Y}}_{=: Ω¹_{X/T}(\log D)} \ar[r] & 0.
    \end{tikzcd}
  \]
  If $t ∈ T$ is any point with fibre $X_t := φ^{-1}(t)$ and inclusion $ι_t : X_t
  → X$, then $X_t$ is smooth and $D_t := ι^*_t D$ is logarithmic and snc.  There
  exists a natural identification between restrictions rendering the following
  diagram commutative,
  \[
    \begin{tikzcd}[row sep=1cm, column sep=1.2cm]
      ι^* Ω¹_X(\log D) \ar[r, two heads, "ι^*_t q"] \ar[d, equal] & ι^*_t Ω¹_{X/T}(\log D) \ar[d, hook, two heads, "\text{ident.}"] \\
      ι^* Ω¹_X(\log D) \ar[r, two heads, "dι_t"'] & Ω¹_{X_t}(\log D_t)
    \end{tikzcd}
  \]
\end{reminder}

\subsection{Forms on fibre spaces}
\approvals{Erwan & yes \\ Frédéric & yes \\ Stefan & yes}

We recall (and slightly generalize) a fundamental fact of Kähler geometry: If $φ
: X → T$ is a fibration and $σ$ a closed 1-form on $X$ that vanishes on one
fibre, then $σ$ comes from $T$.

\begin{prop}\label{prop:2-21}%
  Let $(X,D)$ be a logarithmic snc pair and let $φ : X \twoheadrightarrow T$ be
  an snc morphism of $(X,D)$.  Assume that $X$ is Kähler and that $φ$ is proper.
  If $σ ∈ H⁰(X, Ω¹_X(\log D))$ is closed, then the following statements are
  equivalent.
  \begin{enumerate}
  \item\label{il:2-21-1} There exists one point $t ∈ T$ with fibre $X_t :=
    φ^{-1}(t)$ and inclusion $ι_t : X_t → X$ such that the restriction of $σ$ to
    $X_t$ vanishes,
    \[
      dι_t σ = 0 ∈ H⁰\bigl(X_t,\ Ω¹_{X_t} (\log D_t)\bigr).
    \]

  \item\label{il:2-21-2} Locally on $X$, the form $σ$ comes from downstairs: $σ
    ∈ H⁰(X, φ^* Ω¹_T)$.
  \end{enumerate}
\end{prop}
\begin{proof}
  The direction \ref{il:2-21-2} $⇒$ \ref{il:2-21-1} is trivial.  To prove the
  converse, assume that a closed form $σ$ and a point $t ∈ T$ with the
  properties of \ref{il:2-21-1} are given.  We need to show that
  \[
    σ ∈ H⁰(X, φ^* Ω¹_T)
    \overset{\text{Reminder~\ref{rem:2-20}}}{⇔}
    dι_s σ = 0 ∈ H⁰\bigl(X_s,\ Ω¹_{X_s} (\log D_s)\bigr), \text{ for every } s ∈ T.
  \]
  Assume that a point $s ∈ T$ is given.  Since $(X_s, D_s)$ is snc and $X_s$ is
  compact and Kähler, recall that the image of the natural integration map,
  \[
    π_1(X_s ∖ D_s) → H⁰\bigl(X_s,\ Ω¹_{X_s} (\log D_s)\bigr)^*, \quad γ ↦ \left(\textstyle τ ↦ \int_{γ} τ\right)
  \]
  spans the vector space $H⁰\bigl(X_s,\ Ω¹_{X_s} (\log D_s)\bigr)^*$, see
  \cite[Thm.~4.5.4]{MR3156076}.  To prove that $dι_s σ$ vanishes, it will
  therefore suffice to prove that
  \[
    \int_{γ_s} dι_s σ = 0, \quad \text{for every loop $γ_s$ in $X_s ∖ D_s$.}
  \]
  Assume that a loop $γ_s$ in $X_s ∖ D_s$ is given and recall from
  Reminder~\ref{remi:2-19} that $φ : X ∖ D \twoheadrightarrow T$ is a
  differentiable fibre bundle over a path connected base.  The homotopy exact
  sequence therefore implies that the loop $γ_s$ is homotopy equivalent within
  $X ∖ D$ to a loop $γ_t$ in $X_t ∖ D_t$.  We have
  \[
    \int_{γ_s} dι_s σ
    = \int_{ι_s◦γ_s} σ
    \overset{(*)}{=} \int_{ι_t◦γ_t} σ
    = \int_{γ_t} dι_t σ
    \overset{\text{\ref{il:2-21-1}}}{=} \int_{γ_t} 0 = 0
  \]
  where $(*)$ follows from homotopy equivalence of $γ_s$ and $γ_t$, and
  closedness of the holomorphic form $σ$.
\end{proof}

\subsection{Foliations defined by meromorphic maps}
\approvals{Erwan & yes \\ Frédéric & yes \\ Stefan & yes}

The conclusion of Proposition~\ref{prop:2-21} is frequently summarized by saying
that ``the logarithmic form $σ$ annihilates the foliation defined by $φ$''.  To
avoid confusion, we define this terminology explicitly.

\begin{defn}[Foliation defined by meromorphic maps]\label{def:2-22}%
  If $φ : X \dasharrow Y$ is a meromorphic map between complex manifolds, let $ι
  : X° ↪ X$ be the inclusion of the maximal open set where $φ$ is well-defined.
  We refer to
  \begin{equation}\label{eq:2-22-1}%
    ι_* \ker \Bigl( T(φ|_{X°}) : 𝒯_{X°} → (φ|_{X°})^* 𝒯_Y \Bigr) ⊆ ι_* 𝒯_{X°} = 𝒯_X
  \end{equation}
  as the \emph{foliation defined by $φ$}.  Observe that this sheaf is coherent,
  saturated as a subsheaf of $𝒯_X$ and hence reflexive.  It is closed under the
  Lie bracket.
\end{defn}

\begin{notation}[Foliation defined by meromorphic maps]
  Assume the setting of Definition~\ref{def:2-22}.  If no confusion is likely to
  arise, we write $\ker (Tφ) ⊆ 𝒯_X$ for the foliation defined by $φ$, as a
  shorthand for the more cumbersome expression \eqref{eq:2-22-1}.
\end{notation}

\begin{defn}[Logarithmic forms annihilating foliations]\label{def:2-24}%
  If $(X, D)$ is an snc pair, if $ℱ ⊆ 𝒯_X$ is a foliation and $σ ∈ H⁰\bigl(
  X,\, Ω¹_X(\log D) \bigr)$ a logarithmic form, we say that ``$σ$ annihilates
  $ℱ$'' if the natural sheaf morphism
  \[
    σ : ℱ → 𝒪_X(D_{\red})
  \]
  is constantly zero.
\end{defn}

%
%
\svnid{$Id: 03-ggenerated.tex 110 2025-12-11 12:34:48Z kebekus $}
\selectlanguage{british}

\section{Rational maps induced by generically generated sheaves}
\subversionInfo
\approvals{Erwan & yes \\ Frédéric & yes \\ Stefan & yes}
\label{sec:3}

To prepare for the discussion of ``Bogomolov sheaves'' defined by strict wedge
subspaces in Section~\ref{sec:6}, this section considers sheaves $ℰ$, subspaces
$E ⊆ H⁰\bigl( X,\, ℰ \bigr)$ and rational maps coming from the induced linear
systems $Λ^{\rank ℰ} E → H⁰\bigl( X,\, \det ℰ \bigr)$.

\begin{defn}[Sheaves generically generated by spaces of sections]\label{def:3-1}%
  Let $X$ be a compact and normal analytic variety, and let $ℰ$ be a coherent
  sheaf of $𝒪_X$-modules.  Assume that $ℰ$ is not a torsion sheaf, and let $E ⊆
  H⁰\bigl( X,\, ℰ \bigr)$ be a linear subspace.  Call $ℰ$ \emph{generically
  generated by sections in $E$} if the following two equivalent conditions hold.
  \begin{enumerate}
  \item The natural map $E ⊗ 𝒪_X → ℰ$ is generically surjective.
  
  \item The natural map $λ_E : Λ^{\rank ℰ} E → H⁰\bigl( X,\, \det ℰ \bigr)$ is
    non-trivial.
  \end{enumerate}
\end{defn}

The following construction is the basis for all that follows in this section.

\begin{construction}[Projection to linear systems induced by spaces of sections]\label{cons:3-2}%
  In the setup of Definition~\ref{def:3-1}, assume that $ℰ$ is generically
  generated by sections in $E$.  Observe that $\det ℰ$ is reflexive with a
  non-trivial space of sections.  Following the notation of
  Section~\ref{sec:2-3}, denote the associated rational map by
  \[
    φ_{\det ℰ} : X \dasharrow ℙ \Bigl( H⁰\bigl( X,\, \det ℰ \bigr) \Bigr).
  \]
  As remarked in \ref{rem:2-11}, the inclusion $\img λ_E ⊆ H⁰\bigl( X,\, \det ℰ
  \bigr)$ of linear spaces induces a rational projection map between the
  projectivizations,
  \[
    ℙ \Bigl( H⁰\bigl( X,\, \det ℰ \bigr) \Bigr) \dashrightarrow ℙ \Bigl( \img λ_E \Bigr).
  \]
  Given that elements of $\img λ_E$ do not vanish at general points of $X$, the
  image of $φ_{\det ℰ}$ is not contained in the indeterminacy of the projection,
  and we obtain a composed rational map $η_E := φ_{\img λ_E, \det ℰ}$ as
  follows,
  \[
    \begin{tikzcd}[column sep=2cm]
      X \arrow[r, dashed, "φ_{\det ℰ}"'] \arrow[rr, dashed, "η_E", bend left=15] & ℙ \Bigl( H⁰\bigl( X,\, \det ℰ \bigr) \Bigr)
      \arrow[r, dashed, two heads, "\text{projection}"'] & ℙ \Bigl( \img λ_E \Bigr).
    \end{tikzcd}
  \]
\end{construction}

\begin{notation}[Projection to linear systems induced by spaces of sections]
  We will use the notation $η_E$ of Construction~\ref{cons:3-2} throughout the
  present paper, whenever we discuss sheaves generically generated by spaces of
  sections.
\end{notation}

\begin{rem}[Sheaves of differentials generically generated by spaces of sections]\label{rem:3-4}%
  Consider the setup of Construction~\ref{cons:3-2} in case $X$ is a Kähler
  manifold, $D$ an snc divisor on $X$ and $ℰ ⊆ Ω¹_X(\log D)$ a sheaf of
  logarithmic differentials.  Then, there exists a sequence of inequalities,
  \[
     \dim \img η_E ≤ \dim \img φ_{\det ℰ} ≤ κ(\det ℰ) ≤ \rank ℰ,
  \]
  where the last inequality is given by the classic vanishing theorem of
  Bogomolov-Sommese\footnote{See \cite[Cor.~6.9]{EV92} for the projective case
  and \cite[Sect.~6]{MR4897417} for a discussion of the vanishing theorem in the
  non-projective Kähler setting}.  Recall that $\det ℰ$ is called a
  \emph{Bogomolov sheaf} if the equality $κ(\det ℰ) = \rank ℰ$ holds.
\end{rem}

\subsection{Functoriality}
\approvals{Erwan & yes \\ Frédéric & yes \\ Stefan & yes}

The following proposition asserts that projection to linear systems induced by
spaces of sections is functorial in inclusions of sheaves generically generated
by spaces of sections.  The proof is tedious, but elementary and certainly not
surprising.  We include full details for completeness' sake.

\begin{prop}[Functoriality of $η_•$ in inclusions]\label{prop:3-5}%
  Let $X$ be a compact and normal analytic variety and let $ℱ ⊆ ℰ$ be an
  inclusion of torsion free, coherent sheaves of $𝒪_X$-modules.  Assume that
  $ℱ$ and $ℰ$ are generically generated by sections in $F ⊆ H⁰\bigl( X,\, ℱ
  \bigr)$ and $E ⊆ H⁰\bigl( X,\, ℰ \bigr)$, respectively.  If $F ⊆ E$, then
  there exists a commutative diagram of composable rational maps,
  \begin{equation}\label{eq:3-5-1}
    \begin{tikzcd}[column sep=2cm]
      X \ar[r, dashed, "η_E"'] \ar[rr, bend left=15, dashed, "η_F"] & ℙ \Bigl( \img λ_E \Bigr) \ar[r, dashed, two heads, "∃ η_{E,F}"'] & ℙ \Bigl( \img λ_F \Bigr),
    \end{tikzcd}
  \end{equation}
  where $η_{E,F}$ is a linear projection.
\end{prop}

\begin{notation}[Functoriality of $η_•$ in inclusions]
  We will use the notation $η_{E,F}$ of Proposition~\ref{prop:3-5} throughout
  the paper, whenever we discuss inclusions of sheaves generically generated by
  spaces of sections.
\end{notation}

\begin{proof}[Proof of Proposition~\ref{prop:3-5}]\CounterStep%
  Choose a general point $x ∈ X$ and let $τ_1, …, τ_a ∈ E$ be a sequence of
  sections such that the classes of $τ_1(x), …, τ_a(x)$ form a basis of the
  quotient space $ℰ_x/ℱ_x$.  We obtain a linear injection
  \[
    ι_τ : H⁰\bigl( X,\, \det ℱ \bigr) ↪ H⁰\bigl( X,\, \det ℰ \bigr), \quad σ ↦ σ Λ τ_1 Λ ⋯ Λ τ_a.
  \]
  Observing that the injection $ι_τ$ restricts to a linear injection between the
  images of the $λ$-operators, we obtain a commutative diagram of linear
  injections,
  \[
    \begin{tikzcd}
      \img λ_E \ar[r, hook] & H⁰\bigl( X,\, \det ℰ \bigr) \\
      \img λ_F \ar[r, hook] \ar[u, hook, "ι_τ|_{\img λ_F}"] & H⁰\bigl( X,\, \det ℱ \bigr), \ar[u, hook, "ι_τ"']
    \end{tikzcd}
  \]
  and hence a diagram of composable rational maps between the projectivizations,
  \begin{equation}\label{eq:3-7-1}
    \begin{tikzcd}[column sep=2cm, row sep=1cm]
      X \ar[r, dashed, "φ_{\det ℰ}"'] \ar[d, equal] \ar[rr, dashed, "η_E", bend left=15] & ℙ \Bigl( H⁰\bigl( X,\, \det ℰ \bigr) \Bigr) \ar[r, dashed, two heads, "\text{projection}"'] \ar[d, dashed, two heads, "ℙ(ι_τ)"'] & ℙ \Bigl( \img λ_E \Bigr) \ar[d, dashed, two heads, "ℙ(ι_τ|_{\img λ_F})"] \\
      X \ar[r, dashed, "φ_{\det ℱ}"] \ar[rr, dashed, "η_F"', bend right=15] & ℙ \Bigl( H⁰\bigl( X,\, \det ℱ \bigr) \Bigr) \ar[r, dashed, two heads, "\text{projection}"] & ℙ \Bigl( \img λ_F \Bigr),
    \end{tikzcd}
  \end{equation}
  whose right square commutes by construction.  We are done once we prove that
  the left square of \eqref{eq:3-7-1} also commutes.  To this end, let $x ∈ X$
  be a general point.  Identifying points of projective spaces with
  codimension-one subspaces of the underlying vector spaces, we have
  \begin{align*}
    φ_{\det ℰ}(x) & = \bigl\{ μ ∈ H⁰\bigl( X,\, \det ℰ \bigr) \::\: μ(x) = 0 \bigr\} \\
    φ_{\det ℱ}(x) & = \bigl\{ σ ∈ H⁰\bigl( X,\, \det ℱ \bigr) \::\: σ(x) = 0 \bigr\} \\
    \intertext{and then}
    ℙ(ι_τ)\bigl(φ_{\det ℰ}(x)\bigr) & = ι_{τ}^{-1}\bigl\{ μ ∈ H⁰\bigl( X,\, \det ℰ \bigr) \::\: μ(x) = 0 \bigr\} \\
    & = \bigl\{ σ ∈ H⁰\bigl( X,\, \det ℱ \bigr) \::\: (σ Λ τ_1 Λ ⋯ Λ τ_a) (x) = 0 \bigr\}.
  \end{align*}
  But since $x$ is general and since forming a basis is a Zariski open property,
  the choice of $τ_•$ guarantees that
  \[
    σ(x) = 0 ⇔ (σ Λ τ_1 Λ ⋯ Λ τ_a) (x) = 0,
    \quad \text{for every } σ ∈ H⁰\bigl( X,\, \det ℱ \bigr),
  \]
  which is the desired commutativity statement.
\end{proof}

For later use in Section~\ref{sec:6}, we remark that Proposition~\ref{prop:3-5}
applies in a setting where the larger sheaf is obtained as a finite sum of
subsheaves.

\begin{cor}[Sums of sheaves generically generated by spaces of sections]\label{cor:3-8}%
  Let $X$ be a projective manifold and let $ℱ_1, …, ℱ_a ⊆ ℰ$ be inclusions of
  torsion free, coherent sheaves of $𝒪_X$-modules.  Assume that the $ℱ_•$ are
  generically generated by sections in $F_• ⊆ H⁰\bigl( X,\, ℱ_• \bigr)$ and set
  \[
    ℱ := \sum_i ℱ_i ⊆ ℰ
    \qquad
    F := \sum_i F_i ⊆ H⁰\bigl( X,\, ℰ \bigr).
  \]
  Then, $ℱ$ is generically generated by sections in $F$ and there exists a
  commutative diagram of composable rational maps,
  \begin{equation}\label{eq:3-8-1}
    \begin{tikzcd}[column sep=2cm]
      X \ar[r, dashed, "η_F"'] \ar[rr, bend left=15, dashed, "η_{F_1} ⨯ ⋯ ⨯ η_{F_a}"] & ℙ \Bigl( \img λ_F \Bigr) \ar[r, dashed, "∃ η"'] & ℙ \Bigl( \img λ_{F_1} \Bigr) ⨯ ⋯ ⨯ ℙ \Bigl( \img λ_{F_a} \Bigr).
    \end{tikzcd}
  \end{equation}
  In particular, we have
  \begin{equation}\label{eq:3-8-2}
    \dim \img η_F ≥ \dim \img η_{F_1} ⨯ ⋯ ⨯ η_{F_a}.
  \end{equation}
\end{cor}
\begin{proof}
  The assertion that $ℱ$ is generically generated by sections in $F$ is clear.
  Apply Proposition~\ref{prop:3-5} to the subsheaves $ℱ_• ⊂ ℱ$ and take $η :=
  η_{F, F_1} ⨯ ⋯ ⨯ η_{F, F_a}$.  Inequality~\eqref{eq:3-8-2} follows from
  commutativity of \eqref{eq:3-8-1}, which guarantees that the image of $η_F$
  dominates the image of $η_{F_1} ⨯ ⋯ ⨯ η_{F_a}$.
\end{proof}

%
%
\svnid{$Id: 04-extension.tex 130 2026-01-07 16:08:21Z kebekus $}
\selectlanguage{british}

\section{Extension of adapted reflexive differentials, Proof of Theorem~\ref*{thm:1-6}}
\label{sec:4}
\subversionInfo
\approvals{Erwan & yes \\ Frédéric & yes \\ Stefan & yes}

Theorem~\ref{thm:1-6} asserts that reflexive 1-forms on dlt $\cC$-pairs extend to
log resolutions of singularities.  Sections~\ref{sec:4-1} and \ref{sec:4-2} make
this statement precise and compare the notion to the ``pull-back'' discussed in
\cite[Sect.~5]{orbiAlb1}.  The subsequent Section~\ref{sec:4-3} provides
elementary extendability criteria, which are then applied in
Section~\ref{sec:4-4} to prove Theorem~\ref{thm:1-6}.

\subsection{Definition}
\label{sec:4-1}
\approvals{Erwan & yes \\ Frédéric & yes \\ Stefan & yes}

The extension problem for adapted reflexive differentials considers a cover
$\what{X}$ of a $\cC$-pair, a resolution $\wtilde{X}$ of the singularities and
asks if every adapted reflexive differential comes from a logarithmic
differential on $\wtilde{X}$.  We consider the following setting throughout the
present section.

\begin{setting}[Extension of adapted reflexive differentials]\label{setting:4-1}%
  Given a $\cC$-pair $(X, D)$, consider sequences of morphisms of the following
  form,
  \begin{equation}\label{eq:4-1-1}
    \begin{tikzcd}[column sep=4cm]
      \wtilde{X} \ar[r, two heads, "π\text{, log resolution of } \bigl(\what{X}\protect{,}\: γ^* ⌊D⌋\bigr)"] & \what{X} \ar[r, "γ\text{, $q$-morphism}"] & X.
    \end{tikzcd}
  \end{equation}
 Write
  \[
    \what{X}° := \what{X}_{\reg} ∩ γ^{-1} (X_{\reg} ∖ \supp D)
  \]
  and observe that $\what{X}°$ is a non-empty, Zariski open subset of
  $\what{X}$.
\end{setting}

\begin{rem}[Sheaves of reflexive log differentials]\label{rem:4-2}%
  Maintain Setting~\ref{setting:4-1}.  If $p ∈ ℕ$ is any number, we consider the
  natural sheaves of reflexive log differentials,
  \begin{align}
    \label{eq:4-2-1} Ω^{[p]}_{(X,D,γ)} & ⊆ Ω^{[p]}_{\what{X}} \bigl(\log γ^* ⌊D⌋ \bigr) \\
    \label{eq:4-2-2} π_* Ω^p_{\wtilde{X}}\bigl(\log π^*γ^* ⌊D⌋ \bigr) & ⊆ Ω^{[p]}_{\what{X}} \bigl(\log γ^* ⌊D⌋\bigr).
  \end{align}
  We refer the reader to \cite[Sect.~4.2]{orbiAlb1} where
  Inclusion~\eqref{eq:4-2-1} is discussed at length.  Inclusion~\eqref{eq:4-2-2}
  exists because the sheaf on the left is torsion free and agrees with the
  reflexive sheaf on the right on the big open set where the resolution map $π$
  is isomorphic.
\end{rem}

\begin{defn}[Extension of adapted reflexive differentials]\label{def:4-3}%
  Let $(X, D)$ be a $\cC$-pair and let $p ∈ ℕ$ be any number.  Say that
  \emph{adapted reflexive $p$-forms on $(X,D)$ extend to log resolutions of
  singularities} if for every sequence of the form \eqref{eq:4-1-1}, we have
  \begin{equation}\label{eq:4-3-1}
    Ω^{[p]}_{(X,D,γ)} ⊆ π_* Ω^p_{\wtilde{X}} \bigl(\log π^*γ^* ⌊D⌋\bigr).
  \end{equation}
\end{defn}

\begin{rem}[Sheaves in Definition~\ref{def:4-3}]
  Recall from Remark~\ref{rem:4-2} that the sheaves on the left and right are
  subsheaves of $Ω^{[p]}_{\what{X}}(\log γ^* ⌊D⌋)$.  Inclusion~\eqref{eq:4-3-1}
  is therefore meaningful.
\end{rem}

The word ``extension'' in Definition~\ref{def:4-3} is justified by the following
remark.

\begin{rem}[Extension of globally defined adapted reflexive differentials]
  Assume that Inclusion~\eqref{eq:4-3-1} of Definition~\ref{def:4-3} holds.  If
  \begin{align*}
    \what{σ} & ∈ H⁰\bigl( \what{X},\, Ω^{[p]}_{(X,D,γ)}\bigr) ⊆ H⁰\bigl( \what{X},\, Ω^{[p]}_{\what{X}}(\log γ^* ⌊D⌋)\bigr), \\
    \intertext{is any adapted reflexive $p$-form, then there exists a logarithmic differential}
    \wtilde{σ} & ∈ H⁰\bigl( \wtilde{X},\, Ω^p_{\wtilde{X}}(\log π^*γ^* ⌊D⌋)\bigr)
  \end{align*}
  that agrees over $(\what{X}, γ^* ⌊D⌋)_{\snc}$ with the pull-back of
  $\what{σ}$.
\end{rem}

\begin{prop}[Independence of the resolution]
  Let $(X, D)$ be a $\cC$-pair and let $p ∈ ℕ$ be any number.  Then, the
  following statements are equivalent.
  \begin{enumerate}
  \item Inclusion~\eqref{eq:4-3-1} holds for every sequence of the form
    \eqref{eq:4-1-1},

  \item Inclusion~\eqref{eq:4-3-1} holds for one sequence of the form
    \eqref{eq:4-1-1}.
  \end{enumerate}
\end{prop}
\begin{proof}
  Given that any two log resolutions are dominated by a common third, the
  subsheaf $π_* Ω^p_{\wtilde{X}}\bigl(\log π^*γ^* ⌊D⌋ \bigr) ⊆
  Ω^{[p]}_{\what{X}} \bigl(\log γ^* ⌊D⌋\bigr)$ does not depend on $π$.
\end{proof}

\subsection{Relation to the literature}
\label{sec:4-2}
\approvals{Erwan & yes \\ Frédéric & yes \\ Stefan & yes}

For locally uniformizable pairs\footnote{Locally uniformizable pairs are
$\cC$-pairs with particularly simple singularities, akin to quotients of pairs
with reduced, normal crossing boundary divisor.  We refer the reader to
\cite[Sect.~2.5.1]{orbiAlb1} for the definition and a detailed discussion.},
Section~5 of the reference paper \cite{orbiAlb1} discusses ``pull-back'', a
concept closely related to the ``extension'' of Definition~\ref{def:4-3}.  The
definition of ``pull-back'' replaces sequences of the form \eqref{eq:4-1-1} by
sequences where $π$ is an arbitrary morphism from a smooth space, and replaces
Inclusion~\eqref{eq:4-3-1} by the existence of a pull-back morphism ``$d_{\cC}
π$'' that generalizes the standard pull-back of Kähler differentials.  The
following proposition relates these notions for later reference.

\begin{prop}[Extension as a pull-back property]\label{prop:4-7}%
  Let $(X, D)$ be a $\cC$-pair and let $p ∈ ℕ$ be any number.  Adapted reflexive
  $p$-forms on $(X,D)$ extend to log resolutions of singularities if and only if
  for every sequence of the form \eqref{eq:4-1-1}, there exists an injective
  sheaf morphism
  \[
    d_{\cC} π : π^{[*]} Ω^{[p]}_{(X,D,γ)} ↪ Ω^p_{\wtilde{X}}(\log π^*γ^* ⌊D⌋)
  \]
  that agrees on the Zariski open set $π^{-1} (\what{X}°)$ with the standard
  pull-back of Kähler differentials.
\end{prop}

Proposition~\ref{prop:4-7} will be shown below.  The phrase ``that agrees […]
with the standard pull-back of Kähler differentials'' might require an
explanation.

\begin{explanation}[Agreeing with the standard pull-back of Kähler differentials]\label{exp:4-8}%
  In the setup of Proposition~\ref{prop:4-7}, consider the non-trivial open set
  \[
    \what{X}° = \what{X}_{\reg} ∩ γ^{-1} (X_{\reg} ∖ \supp D)
  \]
  and recall from \cite[Ex.~4.6]{orbiAlb1} that on $\what{X}°$, the sheaf of
  adapted reflexive differentials take the simple form
  \[
    Ω^{[1]}_{(X,D,γ)}|_{X°} = γ^* Ω¹_{X_{\reg}}
  \]
  In particular, $Ω^{[1]}_{(X,D,γ)}|_{X°}$ is a subsheaf of the sheaf of Kähler
  differentials $Ω¹_{\what{X}°}$.  A sheaf morphism $d_{\cC} π$ agrees on
  $π^{-1} (\what{X}°)$ with the standard pull-back of Kähler differentials if
  there exists a commutative diagram of sheaves on $π^{-1} (\what{X}°)$,
  \[
  \begin{tikzcd}[column sep=2cm]
    \left.π^{[*]} Ω^{[1]}_{(X,D,γ)} \right|_{π^{-1} (\what{X}°)} \ar[r, hook, "d_{\cC} π"] \ar[d, hook] & \left.  Ω¹_{\wtilde{X}}(\log \wtilde D) \right|_{π^{-1} (\what{X}°)} \ar[d, equal] \\
    \left(π|_{π^{-1} (\what{X}°)}\right)^* Ω¹_{\what{X}°} \ar[r, hook, "d (π|_{π^{-1} (\what{X}°)})", "\text{standard pull-back}"'] & Ω¹_{π^{-1} (\what{X}°)}.
  \end{tikzcd}
  \]
\end{explanation}

\begin{proof}[Proof of Proposition~\ref{prop:4-7}]
  Assume that a sequence of the form \eqref{eq:4-1-1} is given.  If adapted
  reflexive $p$-forms on $(X,D)$ extend, the pull-back of
  Inclusion~\eqref{eq:4-3-1} gives a morphism
  \[
    π^* Ω^{[p]}_{(X,D,γ)} ⊆ π^* π_* Ω^p_{\wtilde{X}} \bigl(\log π^*γ^* ⌊D⌋\bigr) → Ω^p_{\wtilde{X}} \bigl(\log π^*γ^* ⌊D⌋\bigr),
  \]
  which factors via the reflexive hull $π^{[*]} Ω^{[p]}_{(X,D,γ)}$ because the
  sheaf on the right is locally free.  Conversely, if a sheaf morphism of the
  form $d_{\cC} π$ is given, we have inclusions
  \begin{align*}
    Ω^{[p]}_{(X,D,γ)} & ⊆ π_* π^* Ω^{[p]}_{(X,D,γ)} ⊆ π_* π^{[*]} Ω^{[p]}_{(X,D,γ)} \\
    & ⊆ π_* Ω^p_{\wtilde{X}}(\log π^*γ^* ⌊D⌋) && \text{Inclusion }π_* d_{\cC} π \\
    & ⊆ Ω^{[p]}_{\what{X}} \bigl(\log γ^* ⌊D⌋\bigr) && \text{Inclusion~\eqref{eq:4-2-2}}
  \end{align*}
  as required in Definition~\ref{def:4-3}.
\end{proof}

\begin{cor}[Extension of adapted reflexive form on locally uniformizable pairs]
  Let $(X, D)$ be a $\cC$-pair and let $p ∈ ℕ$ be any number.  If $(X,D)$ is
  locally uniformizable, then adapted reflexive $p$-forms on $(X,D)$ extend to
  log resolutions of singularities.
\end{cor}
\begin{proof}
  Given any sequence of the form \eqref{eq:4-1-1}, recall from
  \cite[Sect.~5]{orbiAlb1} that there exist injective sheaf morphisms
  \[
    d_{\cC} π : π^{[*]} Ω^{[p]}_{(X,D,γ)} ↪ Ω^p_{\wtilde{X}}(\log π^*γ^* ⌊D⌋)
  \]
  that agree on $π^{-1} (\what{X}°)$ with the standard
  pull-back of Kähler differentials.
\end{proof}

\subsection{Extension Criteria}
\label{sec:4-3}
\approvals{Erwan & yes \\ Frédéric & yes \\ Stefan & yes}

This section establishes criteria to guarantee extension of adapted reflexive
differentials.  To begin, we note that extendability is a local property.  The
elementary proof is left to the reader.

\begin{prop}[Local nature of the extension problem]\label{prop:4-10} Let $(X, D)$
  be a $\cC$-pair, let $p ∈ ℕ$ be a number and let $(U_α)_{α ∈ A}$ be a covering
  of $X$ by sets that are open in the analytic topology.  Then, the following
  statements are equivalent.
  \begin{enumerate}
  \item Adapted reflexive $p$-forms on $(X,D)$ extend to log resolutions of
    singularities.

  \item Adapted reflexive $p$-forms on $(U_α,D|_{U_α})$ extend to log
    resolutions of singularities, for every $α ∈ A$.  \qed
  \end{enumerate}
\end{prop}

As an almost immediate corollary of the stratified extension results of
Appendix~\ref{sec:A}, we find that adapted reflexive differentials extend if they
extend outside a high-codimension subset.

\begin{prop}[Restriction to very big open sets]\label{prop:4-11}%
  Let $(X, D)$ be a $\cC$-pair, let $p ∈ ℕ$ be a number and let $Z ⊂ X$ be an
  analytic subset of $\codim_X Z ≥ p+2$.  Then, the following statements are
  equivalent.
  \begin{enumerate}
  \item\label{il:4-11-1} Adapted reflexive $p$-forms on $(X,D)$ extend to log
    resolutions of singularities.

  \item\label{il:4-11-2} Adapted reflexive $p$-forms on $(X∖ Z,D|_{X∖ Z})$
    extend to log resolutions of singularities.
  \end{enumerate}
\end{prop}
\begin{proof}
  Only the implication \ref{il:4-11-2}$⇒$\ref{il:4-11-1} is interesting.
  Assuming that adapted reflexive $p$-forms on $(X∖ Z,D|_{X∖ Z})$ extend to log
  resolutions of singularities, consider a sequence of morphisms as in
  \eqref{eq:4-1-1} of Setting~\ref{setting:4-1}.  We need to check
  Inclusion~\eqref{eq:4-3-1} of Definition~\ref{def:4-3}.  To this end, consider
  the inclusion $ι : γ^{-1} (X ∖ Z) ↪ \what{X}$.  Assumption~\ref{il:4-11-2}
  will then give inclusions between sheaves
  \[
    ι^* Ω^{[p]}_{(X,D,γ)} ⊆ ι^* π_* Ω^p_{\wtilde{X}}(\log π^*γ^* ⌊D⌋) ⊆ ι^* Ω^{[p]}_{\what{X}}(\log γ^* ⌊D⌋).
  \]
  Since the push-forward $ι_*$ preserves inclusions, we obtain
  \begin{equation}\label{eq:4-11-3}
    ι_* ι^* Ω^{[p]}_{(X,D,γ)} ⊆ ι_* ι^* π_* Ω^p_{\wtilde{X}}(\log π^*γ^* ⌊D⌋) ⊆ ι_* ι^* Ω^{[p]}_{\what{X}}(\log γ^* ⌊D⌋).
  \end{equation}
  But reflexivity shows that
  \begin{align*}
    ι_* ι^* Ω^{[p]}_{(X,D,γ)} & = Ω^{[p]}_{(X,D,γ)} \\
    ι_* ι^* Ω^{[p]}_{\what{X}}(\log γ^* ⌊D⌋) & = Ω^{[p]}_{\what{X}}(\log γ^* ⌊D⌋) \\
    \intertext{and Corollary~\ref{cor:A-2} of Appendix~\ref{sec:A} gives}
    ι_* ι^* π_* Ω^p_{\wtilde{X}}(\log π^*γ^* ⌊D⌋) & = π_* Ω^p_{\wtilde{X}}(\log π^*γ^* ⌊D⌋).
  \end{align*}
  This presents \eqref{eq:4-11-3} as a reformulation of the desired
  Inclusion~\eqref{eq:4-3-1}.
\end{proof}

Taken at face value, Definition~\ref{def:4-3} (``Extension to log resolutions of
singularities'') required us to check every sequence of the form
\eqref{eq:4-1-1}.  The following proposition simplifies the task, showing that
it suffices to consider sequence where the $q$-morphism $γ$ is adapted.

\begin{prop}[Restriction to adapted $q$-morphisms]\label{prop:4-12}%
  Let $(X, D)$ be a $\cC$-pair and let $p ∈ ℕ$ be any number.  Assume that for
  every sequence of morphisms of the following form,
  \[
    \begin{tikzcd}[column sep=4cm]
      \wtilde{X} \ar[r, two heads, "ρ\text{, log resolution of } (\what{X}\protect{,} δ^* ⌊D⌋)"] & \what{X} \ar[r, two heads, "δ\text{, adapted $q$-morphism}"] & X,
    \end{tikzcd}
  \]
  there is an inclusion
  \begin{equation}\label{eq:4-12-1}
    Ω^{[p]}_{(X,D,δ)} ⊆ ρ_* Ω^p_{\wtilde{X}}(\log ρ^*δ^* ⌊D⌋).
  \end{equation}
  Then, adapted reflexive $p$-forms on $(X,D)$ extend to log resolutions of
  singularities.
\end{prop}
\begin{proof}
  Recall from Proposition~\ref{prop:4-10} that the extension property is local on
  $X$.  We may therefore assume without loss of generality that $X$ is Stein and
  that the divisor $D$ has only finitely many components.  By
  \cite[Lem.~2.36]{orbiAlb1}, this implies the existence of an adapted cover,
  say $μ : \what{W} \twoheadrightarrow X$.

  In order to verify the conditions spelled out in Definition~\ref{def:4-3},
  assume that a sequence of morphism is given as in \eqref{eq:4-1-1} of
  Setting~\ref{setting:4-1}.  Next, construct a diagram of the following form,
  \[
    \begin{tikzcd}[column sep=4cm, row sep=1cm]
      \wtilde{Y} \ar[r, two heads, "ρ\text{, equivariant log resolution}"] \ar[d, two heads, "\wtilde{q}\text{, gen.~finite}"'] & \what{Y} \ar[r, "δ\text{, adapted $q$-morphism}"] \ar[d, two heads, "q\text{, Galois cover}"] & X \ar[d, equal] \\
      \wtilde{X} \ar[r, two heads, "π\text{, log resolution}"'] & \what{X} \ar[r, "γ\text{, $q$-morphism}"'] & X.
    \end{tikzcd}
  \]
  To spell the construction out in detail:
  \begin{itemize}
  \item In order to construct the right square, consider the fibre product
    $\what{W} ⨯_X \what{X}$ and take $\what{Y}$ as the Galois closure over
    $\what{X}$.  The natural morphism $δ$ will then factor via $μ$ and is
    then, by \cite[Obs.~2.27]{orbiAlb1}, itself adapted.  The natural morphism
    $q$ is a cover because $μ$ is.  Denote the Galois group by $G$.
  
  \item In order to construct the left square, let $ρ$ be a $G$-equivariant
    log resolution of the $G$-variety $\wtilde{X} ⨯_{\what{X}} \what{Y}$.
  \end{itemize}
  The desired inclusion~\eqref{eq:4-3-1} will now follow from standard
  considerations involving $G$-invariant push-forward:
  \begin{align*}
    Ω^{[p]}_{(X,D,γ)} & = \left(q_* Ω^{[p]}_{(X,D,δ)}\right)^G && \text{by \cite[Lem.~4.20]{orbiAlb1}} \\
    & ⊆ \left(q_* ρ_* Ω^p_{\wtilde{Y}}(\log ρ^* δ^* ⌊D⌋)\right)^G && \text{Assumption~\eqref{eq:4-12-1}} \\
    & = \left(π_* \wtilde{q}_* Ω^p_{\wtilde{Y}}(\log ρ^* δ^* ⌊D⌋)\right)^G && \text{Commutativity} \\
    & = π_* \left(\wtilde{q}_* Ω^p_{\wtilde{Y}}(\log ρ^* δ^* ⌊D⌋)\right)^G && \text{$G$-invariance} \\
    & = π_* Ω^p_{\wtilde{X}}(\log π^* γ^* ⌊D⌋)
  \end{align*}
  The last equality is a standard fact of the geometry of logarithmic pairs, but
  also follows from \cite[Lem.~4.20]{orbiAlb1}, using that $ρ^* δ^* ⌊D⌋$ and
  $\wtilde{q}^{\,*} π^* γ^* ⌊D⌋$ have identical support.
\end{proof}

We conclude with a proposition showing extension of adapted reflexive $p$-forms
on pairs admitting particularly nice covers, called ``perfectly adapted'' in
\cite[Def.~10]{Nunez24}.

\begin{prop}[Extension on existence of good covering spaces]\label{thm:4-13}%
  Let $(X, D)$ be a $\cC$-pair and let $p ∈ ℕ$ be any number.  Assume that there
  exists a strongly adapted cover $γ : \what{X} \twoheadrightarrow X$ where
  $\what{X}$ is smooth and $γ^* ⌊D⌋$ has snc support.  Then, adapted reflexive
  $p$-forms on $(X,D)$ extend to log resolutions of singularities.
\end{prop}
\begin{proof}
  Using the existence of $γ$, Núñez has shown in \cite[Lem.~21]{Nunez24} that
  the conditions of Proposition~\ref{prop:4-12} hold.
\end{proof}

\subsection{Proof of Theorem~\ref{thm:1-6}}
\label{sec:4-4}
\approvals{Erwan & yes \\ Frédéric & yes \\ Stefan & yes}
\CounterStep

Proposition~\ref{prop:4-11} allows removing algebraic subsets $Z ⊂ X$ of
codimension $\codim_X Z ≥ 3$.  By \cite[Prop.~6.19]{Gra13}, we can therefore
assume that the following dichotomy holds for every point $x ∈ X$.
\begin{enumerate}
\item The point $x$ is contained in $X° := X ∖ \supp ⌊D⌋$.

\item There exists an algebraic neighbourhood $U = U(x) ⊆ X$ and strongly
  adapted cover $γ : \what{U} \twoheadrightarrow U$ where $\what{U}$ is smooth
  and $γ^* ⌊D⌋$ has snc support.
\end{enumerate}
Recalling from Proposition~\ref{prop:4-10} that the extension problem is local, we may
assume that one of the following holds.
\begin{enumerate}
\item\label{il:4-14-3} We have $⌊D⌋ = 0$.

\item\label{il:4-14-4} There exists a strongly adapted cover $γ : \what{X}
  \twoheadrightarrow X$ where $\what{X}$ is smooth and $γ^* ⌊D⌋$ has snc
  support.
\end{enumerate}
In Case~\ref{il:4-14-3}, Núñez has shown in \cite[Thm.~1]{Nunez24} that the
assumptions of Proposition~\ref{prop:4-12} hold.  In Case~\ref{il:4-14-4}, the
claim follows from Proposition~\ref{thm:4-13}.  \qed

%
%
\svnid{$Id: 05-gsheaves.tex 109 2025-12-11 11:37:43Z kebekus $}
\selectlanguage{british}

\section{Bogomolov sheaves and linear systems in reflexive $G$-sheaves}
\subversionInfo

\subsection{The Kodaira dimension for sheaves of adapted reflexive differentials}
\approvals{Erwan & yes \\ Frédéric & yes \\ Stefan & yes}

Before presenting the main result of the section in Proposition~\ref{prop:5-1}
below, we recall the notion of ``$\cC$-Kodaira-Iitaka'' dimension in brief for
the reader's convenience.  Full details are found in \cite[Sects.~4 and
6.2]{orbiAlb1}.

\subsubsection{Tensors on manifolds}
\approvals{Erwan & yes \\ Frédéric & yes \\ Stefan & yes}

If $X$ is a manifold, classic geometry considers the symmetric algebra of
tensors on $X$.  Technically speaking, one considers sheaves $\Sym^n Ω^p_X$
together with symmetric product maps.  In particular, if $ℒ ⊆ Ω^p_X$ is
saturated of rank one, then $ℒ$ is invertible and the symmetric product sheaf
$\Sym^n ℒ$ is a saturated subsheaf of $\Sym^n Ω^p_X$.

\subsubsection{Adapted reflexive tensors on $\cC$-pairs}
\label{sec:5-1-2}
\approvals{Erwan & yes \\ Frédéric & yes \\ Stefan & yes}

If $(X,D)$ is a $\cC$-pair and $γ : \what{X} \twoheadrightarrow X$ is a cover,
the theory of $\cC$-pairs considers the symmetric algebra of ``adapted reflexive
tensors'' on $\what{X}$.  Technically speaking, one considers sheaves
$\Sym^{[n]}_{\cC} Ω^{[p]}_{(X,D,γ)}$ together with symmetric product maps.  In
particular, if $ℒ ⊆ Ω^{[p]}_{(X,D,γ)}$ is saturated of rank one, then $ℒ$ is
reflexive and $\Sym^{[n]} ℒ$ is a subsheaf $\Sym^{[n]}_{\cC} Ω^{[p]}_{(X,D,γ)}$.
In contrast to the manifold setting, this sheaf need not be saturated however,
and one defines the $\cC$-product sheaf
\[
  \Sym^{[n]}_{\cC} ℒ ⊆ \Sym^{[n]}_{\cC} Ω^{[p]}_{(X,D,γ)}
\]
as the saturation.  One can then consider the set
\begin{align*}
  M & := \left\{ m ∈ ℕ\, \left|\, h⁰\bigl( \what{X},\, \Sym^{[m]}_{\cC} ℒ \bigr) > 0 \right.  \right\} \\
  \intertext{and define the \emph{$\cC$-Kodaira-Iitaka dimension} as}
  κ_{\cC}(ℱ) & :=
  \left\{
  \begin{matrix*}[l]
    \max_{m ∈ M} \left\{ \dim \overline{\img φ_{\Sym^{[m]}_{\cC} ℒ}}\right\} & \text{if } M \ne ∅ \\
    -∞ & \text{if } M = ∅,
  \end{matrix*}
  \right.
\end{align*}
where $\overline{•}$ denotes Zariski closure in $ℙ^•$.

\subsection{Comparing $\cC$-Kodaira dimensions for sheaves on different covers}
\approvals{Erwan & yes \\ Frédéric & yes \\ Stefan & yes}

Every\-thing said in Section~\ref{sec:5-1-2} depends on the choice of the cover
$γ$.  Given that any two covers are dominated by a common third, the main result
of the present section explains what happens under a change of covers.

\begin{prop}[$\cC$-Kodaira dimension of sheaves of adapted differentials]\label{prop:5-1}%
  Let $(X,D)$ be a $\cC$-pair, where $X$ is compact.  Assume we are given a
  sequence of covers,
  \[
    \begin{tikzcd}[column sep=2.4cm]
      \what{X}_2 \ar[r, two heads, "α\text{, Galois cover}"'] \ar[rr, bend left=15, two heads, "γ\text{, cover}"] & \what{X}_1 \ar[r, two heads, "β\text{, cover}"'] & X,
    \end{tikzcd}
  \]
  where $α$ is Galois with group $G$.  If $p, d$ are any two numbers, then the
  following statements are equivalent.
  \begin{enumerate}
  \item\label{il:5-1-1} There exists a reflexive sheaf $ℱ_1 ⊆
    Ω^{[p]}_{(X,D,β)}$ of rank one with $κ_{\cC}(ℱ_1) ≥ d$.

  \item\label{il:5-1-2} There exists a reflexive $G$-subsheaf $ℱ_2 ⊆
    Ω^{[p]}_{(X,D,γ)}$ of rank one with $κ_{\cC}(ℱ_2) ≥ d$.
  \end{enumerate}
\end{prop}

Proposition~\ref{prop:5-1} will be shown in Section~\vref{sec:5-4}.  As an
immediate corollary, we obtain a slight generalization of Graf's version of the
Bogomolov-Sommese vanishing theorem, \cite[Thm.~1.2]{Gra13}.  Following the
literature, we extend the notion of a ``Bogomolov sheaf'' to this context.

\begin{cor}[Bogomolov-Sommese vanishing for $G$-sheaves of adapted differentials]\label{cor:5-2}%
  Let $(X,D)$ be a log-canonical $\cC$-pair where $X$ is compact Kähler.  If $γ
  : \what{X} \twoheadrightarrow X$ is any cover that is Galois with group $G$,
  if $p$ is any number and $ℱ_1 ⊆ Ω^{[p]}_{(X,D,γ)}$ any $G$-subsheaf of rank
  one, then $κ_{\cC}(ℱ_1) ≤ p$.
\end{cor}
\begin{proof}
  Apply Proposition~\ref{prop:5-1} in case where $β = \Id_X$ and recall from
  \cite[Thm.~6.14]{orbiAlb1} that if $ℱ_2 ⊆ Ω^{[p]}_{(X, D, \Id_X)}$ is coherent of
  rank one, then $κ_{\cC}(ℱ_2) ≤ p$.
\end{proof}

\begin{defn}[\protect{Bogomolov $G$-sheaf, cf.~\cite[Def.~6.15]{orbiAlb1}}]
  In the setting of Corollary~\ref{cor:5-2}, call $ℱ_1$ a Bogomolov $G$-sheaf if
  the equality $κ_{\cC}(ℱ_1) = p$ holds.
\end{defn}

The existence of Bogomolov $G$-sheaves ties in with the notion of ``special''
$\cC$-pairs.  We refer the reader to \cite[Def.~6.16]{orbiAlb1} for the
definition used here and for references to Campana's original work.

\begin{cor}[Special pairs have no Bogomolov $G$-sheaves]\label{cor:5-4}%
  Let $(X,D)$ be a special $\cC$-pair.  If $γ : \what{X} \twoheadrightarrow X$
  is a Galois cover and if $p$ is any number, then there are no Bogomolov
  $G$-sheaves in $Ω^{[p]}_{(X,D,γ)}$.  \qed
\end{cor}

Note that the definition of ``special'' in \cite[Def.~6.16]{orbiAlb1} implies
that $X$ is compact Kähler and that $(X,D)$ is log-canonical.

\subsection{Linear systems in reflexive $G$-sheaves}
\CounterStep
\approvals{Erwan & yes \\ Frédéric & yes \\ Stefan & yes}

To prepare for the proof of Proposition~\ref{prop:5-1}, the following
Proposition~\ref{prop:5-6} considers a Galois cover $q : X \twoheadrightarrow
Y$, a rank-one $G$-sheaf $ℒ$ on $X$ and compares the rational map $φ_ℒ : X
\dasharrow ℙ^•$ to the rational maps induced by the $G$-invariant
push-forward of the reflexive symmetric products,
\begin{equation}\label{eq:5-5-1}
  φ_{\bigl(q_* \Sym^{[n]} ℒ\bigr)^G} : Y \dasharrow ℙ^•,
  \quad\text{for }n ∈ ℕ.
\end{equation}
To make sense of \eqref{eq:5-5-1}, recall that the sheaves $\bigl(q_* \Sym^{[n]}
ℒ\bigr)^G$ have rank one by construction and are reflexive by
\cite[Lem.~A.4]{GKKP11}.  For consistency with the literature we quote,
Proposition~\ref{prop:5-6} speaks about the reflexive symmetric product of $ℒ$.
Since $ℒ$ has rank one, we have $\Sym^{[n]} ℒ = ℒ^{[⊗ n]}$ and could equally
well speak about the reflexive tensor product.

\begin{prop}\label{prop:5-6}%
  Let $X$ be a compact, normal analytic variety and let $G$ be a finite group
  that acts holomorphically on $X$, with quotient $q : X \twoheadrightarrow
  X/G$.  Let $ℒ$ be a torsion free $G$-sheaf of rank one.  If $h⁰ \bigl( X,\, ℒ
  \bigr) > 0$, then there exists a number $n ∈ ℕ⁺$ such that
  \[
    \dim \img φ_ℒ \le \dim \img φ_{\bigl(q_* \Sym^{[n]} ℒ\bigr)^G}.
  \]
\end{prop}
\begin{proof}
  The proof is tedious but elementary, and might well be known to experts
  working in invariant theory.  We include full details as we are not aware of a
  suitable reference.

  \subsubsection*{Step 1: Setup}

  The group $G$ acts linearly on $L=H⁰\bigl( X,\, ℒ
  \bigr)$ and on its projectivization $ℙ(L)$, in a way
  that makes the map $φ_ℒ$ equivariant.  Consider the quotient $q_ℙ : ℙ (L)
  \twoheadrightarrow ℙ (L)/G$, choose a very ample line bundle $ℋ ∈ \Pic
  \bigl(ℙ(L)/G\bigr)$ and take $n ∈ ℕ$ as the unique number satisfying
  $𝒪_{ℙ(L)}(n) = q^*_ℙ\: ℋ$.  We obtain an identification
  \begin{equation}\label{eq:5-6-1}
     H⁰ \bigl( ℙ(L)/G,\, ℋ \bigr) = H⁰ \bigl( ℙ(L) ,\, 𝒪_{ℙ(L)}(n) \bigr)^G ⊆ H⁰ \bigl( ℙ(L) ,\, 𝒪_{ℙ(L)}(n) \bigr)
  \end{equation}
  and hence a commutative diagram of composable meromorphic maps as follows:
  \begin{equation}\label{eq:5-6-2}
    \begin{tikzcd}[column sep=2cm, row sep=1cm]
      X \ar[r, dashed, "φ_ℒ"] \ar[d, two heads, "q\text{, finite}"'] & ℙ(L) \ar[r, hook, "φ_{𝒪_{ℙ(L)}(n)}\text{, finite}"] \ar[d, two heads, "q_{ℙ}\text{, finite}"] & ℙ \Bigl( H⁰ \bigl( ℙ(L) ,\, 𝒪_{ℙ(L)}(n) \bigr) \Bigr) \ar[d, two heads, dashed, "μ\text{, linear projection}"] \\
      X/G \ar[r, dashed, "(φ_ℒ)^G"'] & ℙ(L)/G \ar[r, hook, "φ_{ℋ}\text{, finite}"'] & ℙ \Bigl( H⁰\bigl( ℙ(L)/G,\, ℋ\bigr)\Bigr).
    \end{tikzcd}
  \end{equation}
  Here,
  \begin{itemize}
  \item the meromorphic map $(φ_ℒ)^G$ is the induced map between the quotients,
    given that $φ_ℒ$ is $G$-equivariant,

  \item the meromorphic map $μ$ is the linear projection induced by
    \eqref{eq:5-6-1} above.
  \end{itemize}

  \subsubsection*{Step 2: Identifications}

  The linear spaces of Diagram~\eqref{eq:5-6-2} come with natural
  identifications and inclusions, which allow expressing some compositions as
  maps induced by incomplete linear systems.  As a variant of \eqref{eq:5-6-1},
  we are interested in the identification
  \begin{align}
    \label{eq:5-6-3} H⁰ \bigl( ℙ(L),\, 𝒪_{ℙ(L)}(n) \bigr) & = \Sym^n L ⊆ H⁰ \bigl( X,\, \Sym^{[n]} ℒ\bigr), \\
    \intertext{and its $G$-invariant version,}
    \nonumber H⁰ \bigl( ℙ(L)/G,\, ℋ \bigr) & = H⁰ \bigl( ℙ(L),\, 𝒪_{ℙ(L)}(n) \bigr)^G && \text{by \eqref{eq:5-6-1}} \\
    \label{eq:5-6-4} & = (\Sym^n L)^G && \text{by \eqref{eq:5-6-3}} \\
    \nonumber & ⊆ H⁰ \bigl( X,\, \Sym^{[n]} ℒ\bigr)^G \\
    \nonumber & = H⁰ \Bigl( X/G,\, q_*\bigl(\Sym^{[n]} ℒ\bigr)^G \Bigr).
  \end{align}
  Observe that \eqref{eq:5-6-4} identifies the composed map $φ_ℋ ◦ (φ_ℒ)^G$ of
  Diagram~\eqref{eq:5-6-2} with the map
  \[
    φ_{\bigl(\Sym^n L\bigr)^G,\: q_*\bigl(\Sym^{[n]} ℒ\bigr)^G} : X/G \dasharrow ℙ \left( H⁰ \Bigl( X/G,\, q_*\bigl(\Sym^{[n]} ℒ\bigr)^G \Bigr) \right).
  \]

  \subsubsection*{Step 3: End of proof}

  In summary, we find
  \begin{align*}
    \dim \img φ_ℒ & = \dim \img φ_{ℋ} ◦ q_{ℙ} ◦ φ_ℒ && \text{Finiteness} \\
    & = \dim \img φ_{ℋ} ◦ (φ_{ℋ})^G ◦ q && \text{Diagram~\eqref{eq:5-6-2}} \\
    & = \dim \img φ_{ℋ} ◦ (φ_{ℋ})^G && \text{Finiteness} \\
    & = \dim \img φ_{\bigl(\Sym^n L\bigr)^G,\: q_*\bigl(\Sym^{[n]} ℒ\bigr)^G} && \text{Step~2} \\
    & \le \dim \img φ_{q_*\bigl(\Sym^{[n]} ℒ\bigr)^G} && \text{Linear subsystem} && \qedhere
  \end{align*}
\end{proof}

\subsection{Proof of Proposition~\ref*{prop:5-1}}
\label{sec:5-4}
\CounterStep
\approvals{Erwan & yes \\ Frédéric & yes \\ Stefan & yes}

We consider the two implications separately.

\subsubsection*{Implication \ref{il:5-1-1} $⇒$ \ref{il:5-1-2}}
\approvals{Erwan & yes \\ Frédéric & yes \\ Stefan & yes}

 Given $ℱ_1 ⊆ Ω^{[p]}_{(X,D,β)}$, set $ℱ_2 := α^{[*]}ℱ_1$ and observe that $ℱ_2$
is indeed a $G$-subsheaf of the $G$-sheaf $Ω^{[p]}_{(X,D,γ)}$.  Next, recall
from \cite[Obs.~4.14]{orbiAlb1} that the standard pull-back of Kähler
differentials extends to inclusions
\[
  α^{[*]} \Sym^{[n]}_{\cC} Ω^{[p]}_{(X,D,β)} ↪ \Sym^{[n]}_{\cC} Ω^{[p]}_{(X,D,γ)},
  \quad \text{for every } n ∈ ℕ.
\]
In particular, we find inclusions
\[
  α^{[*]} \Sym^{[n]}_{\cC} ℱ_1 ↪ \Sym^{[n]}_{\cC} α^{[*]}ℱ_1 = \Sym^{[n]}_{\cC} ℱ_2,
  \quad \text{for every } n ∈ ℕ.
\]
It follows that $κ_{\cC}(ℱ_2) ≥ κ_{\cC}(ℱ_1) ≥ d$, as required.  \qed

\subsubsection*{Implication \ref{il:5-1-2} $⇒$ \ref{il:5-1-1}}
\approvals{Erwan & yes \\ Frédéric & yes \\ Stefan & yes}

Given a $G$-subsheaf $ℱ_2 ⊆ Ω^{[p]}_{(X,D,γ)}$, consider the invariant
push-forward sheaves
\[
  ℰ_n := \bigl(α_* \Sym^{[n]}_{\cC} ℱ_2 \bigr)^G, \text{ for every } n ∈ ℕ.
\]
These sheaves come with inclusions
\begin{align}
  \label{eq:5-7-1} ℰ_n & ⊆ \bigl(α_* \Sym^{[n]}_{\cC} Ω^{[p]}_{(X,D,γ)} \bigr)^G = \Sym^{[n]}_{\cC} Ω^{[p]}_{(X,D,β)} && \text{ for every } n ∈ ℕ \\
  \label{eq:5-7-2} ℰ_n & ⊆ \Sym^{[n]}_{\cC} ℰ_1 && \text{ for every } n ∈ ℕ.
\end{align}
The equality in \eqref{eq:5-7-1} is shown in \cite[Lem.~4.20]{orbiAlb1}.
Inclusion \eqref{eq:5-7-2} follows because $\Sym^{[n]}_{\cC} ℰ_1$ is saturated
inside $\Sym^{[n]}_{\cC} Ω^{[p]}_{(X,D,β)}$ by definition, and because the left
and right side of the inclusion agree over the dense, Zariski open set
\[
  \left(α \bigl(\what{X}_{2, \reg} \bigr) ∩ \what{X}_{1,\reg} ∩ β^{-1} \bigl(X_{\reg}\bigr) \right) ∖ β^{-1} \bigl(\supp D\bigr).
\]
With these preparations at hand and taking $ℱ_1 := ℰ_1$, we find numbers $n, m ∈
ℕ$ such that the following holds.
\begin{align*}
  d & \le κ_{\cC} (ℱ_2) && \text{Assumption} \\
  & = \dim \img φ_{\Sym^{[m]}_{\cC} ℱ_2} && \text{Definition~of } κ_{\cC} \\
  & \le \dim \img φ_{α_* \bigl(\Sym^{[n]}\Sym^{[m]}_{\cC} ℱ_2)\bigr)^G} && \text{Proposition~\ref{prop:5-6}} \\
  & \le \dim \img φ_{α_* \bigl(\Sym^{[m · n]}_{\cC} ℱ_2)\bigr)^G} && \text{\cite[Obs.~4.11]{orbiAlb1}} \\
  & \le \dim \img φ_{\Sym^{[m·n]}_{\cC} ℱ_1} && \text{\eqref{eq:5-7-2}} \\
  & \le κ_{\cC} (ℱ_1) && \text{Definition~of } κ_{\cC} && \qed
\end{align*}

%
%
\svnid{$Id: 06-wedge.tex 130 2026-01-07 16:08:21Z kebekus $}
\selectlanguage{british}

\section{Invariant Bogomolov sheaves defined by strict wedge subspaces}
\subversionInfo
\label{sec:6}
\approvals{Erwan & yes \\ Frédéric & yes \\ Stefan & yes}

Let $X$ be a compact Kähler manifold and let $D$ be a reduced snc divisor on
$X$.  Assume that a finite group $G$ acts on $(X,D)$.  Inspired by constructions
introduced in \cite{zbMATH03743427, Catanese91} to generalize the classic
Castelnuovo-De Franchis theorem, we show how an abundance of logarithmic
one-forms on $X$ can be used to construct $G$-invariant Bogomolov sheaves, even
in cases where no $G$-invariant differential exists.

\subsection{Strict wedge subspaces}
\approvals{Erwan & yes \\ Frédéric & yes \\ Stefan & yes}

Strict wedge subspaces, as defined in \cite{Catanese91}, are the key concept of
the present section.

\begin{defn}[\protect{Strict wedge subspaces, \cite[Sect.~2]{Catanese91}}]\label{def:6-1}%
  Let $X$ be a compact complex manifold and let $D ∈ \Div(X)$ be a reduced snc
  divisor.  Let $k ∈ ℕ$ be a number and $V ⊆ H⁰\bigl( X,\, Ω¹_X(\log D) \bigr)$
  a linear subspace.  We call $V$ a \emph{strict $k$-wedge subspace of $H⁰\bigl(
  X,\, Ω¹_X(\log D) \bigr)$} if the following conditions hold.
  \begin{enumerate}
  \item\label{il:6-1-1} The dimension of $V$ is greater than $k$.  In other
    words, $\dim V > k$.

  \item\label{il:6-1-2} The natural map $Λ^k V → H⁰\bigl( X,\, Ω^k_X(\log D)
    \bigr)$ is injective.
    
  \item\label{il:6-1-3} The natural map $Λ^{k+1} V → H⁰\bigl( X,\,
    Ω^{k+1}_X(\log D) \bigr)$ is identically zero.
  \end{enumerate}
\end{defn}

\begin{rem}[Strict wedge subspaces for special values of $k$]\label{rem:6-2}%
  Assume the setting of Definition~\ref{def:6-1}.
  \begin{enumerate}
  \item\label{il:6-2-1} If $\dim V = k+1$, then every element of $Λ^k V$ is a
    pure wedge, and Condition~\ref{il:6-1-2} can be reformulated by saying that
    no $k$-tuple of linearly independent 1-forms in $V$ wedges to zero.

  \item\label{il:6-2-2} If $\dim X ≤ k$, then Condition~\ref{il:6-1-3} holds
    automatically.
  \end{enumerate}
\end{rem}

\begin{notation}[Sheaves of differentials induced by strict wedge spaces]\label{not:6-3}%
  Assume the setting of Definition~\ref{def:6-1}.  If $V$ is a strict $k$-wedge
  subspace of $H⁰\bigl( X,\, Ω¹_X(\log D) \bigr)$, write
  \[
    \sV := \img \Bigl( 𝒪_X ⊗ V → Ω¹_X(\log D) \Bigr) ⊆ Ω¹_X(\log D)
  \]
  for the sheaf of 1-forms that can locally be written as linear combinations of
  forms in $V$.
\end{notation}

\begin{rem}[Sheaves of differentials induced by strict wedge spaces]\label{rem:6-4}%
  Observe that the sheaf $\sV$ of Notation~\ref{not:6-3} is (generically)
  generated by sections in $V$.  Condition~\ref{il:6-1-2} guarantees that its
  rank equals $k$.
\end{rem}

The following proposition guarantees that strict wedge subspaces exist as soon
as there are sufficiently many one-forms.

\begin{prop}[Existence of strict wedge subspaces]\label{prop:6-5}%
  Let $X$ be a compact complex manifold, let $D ∈ \Div(X)$ be a reduced snc
  divisor, and let $V ⊆ H⁰\bigl( X,\, Ω¹_X(\log D) \bigr)$ be a linear subspace
  of dimension $\dim V > \dim X$.  Then, there exists a number $k ∈ ℕ⁺$ and a
  strict $k$-wedge subspace $V' ⊆ V$ of dimension $\dim V' = k+1$.
\end{prop}
\begin{proof}
  Consider the following set of natural numbers,
  \begin{multline*}
    K := \Bigl\{ ℓ ∈ ℕ \::\: ∃ \text{ linearly independent elements } σ_1, …, σ_{ℓ+1} ∈ V \\
    \text{ with vanishing $(ℓ+1)$-form, } σ_1 Λ ⋯ Λ σ_{ℓ+1} = 0 ∈ H⁰\bigl(X,\, Ω^{ℓ+1}_X(\log D) \bigr) \Bigr\} ⊂ ℕ.
  \end{multline*}
  Observe that the assumption $\dim V > \dim X$ guarantees that $K$ is not
  empty.  Let $k$ be the minimal element of $K$, let $σ_1, …, σ_{k+1} ∈ V$ be
  linearly independent elements with vanishing $k+1$-form, $σ_1 Λ ⋯ Λ σ_{k+1} =
  0$, and let $V'$ be the linear subspace spanned by the $σ_{•}$,
  \[
    V' := \langle σ_1, …, σ_k \rangle ⊆ V.
  \]
  The definition of $K$ together with Item~\ref{il:6-2-1} of
  Remark~\ref{rem:6-2} guarantees that $V'$ is indeed a strict $k$-wedge
  subspace, as desired.
\end{proof}

\subsection{Linear systems defined by strict wedge subspaces}
\approvals{Erwan & yes \\ Frédéric & yes \\ Stefan & yes}

 If $V$ a strict wedge subspace, the next proposition describes the foliation
induced by the meromorphic map $η_V$ introduced in Section~\ref{sec:3},
effectively bounding its dimension from below.  This result is key to all that
follows.

\begin{prop}[Linear systems defined by strict wedge subspaces]\label{prop:6-6}%
  Let $X$ be a compact Kähler manifold, let $D ∈ \Div(X)$ be a reduced snc
  divisor, and let $V ⊆ H⁰\bigl( X,\, Ω¹_X(\log D) \bigr)$ be a strict $k$-wedge
  subspace.  Write $V ⊆ H⁰\bigl( X,\, \sV\bigr)$, where $\sV ⊆ Ω¹_X(\log D)$ is
  the sheaf introduced in Notation~\ref{not:6-3}, and consider the meromorphic
  map
  \[
    η_V : X \dasharrow ℙ (\img λ_V)
  \]
  of Construction~\ref{cons:3-2}.  Then, differentials in $V$ annihilate the
  foliation $\ker (η_V) ⊆ 𝒯_X$ defined by $η_V$.
\end{prop}

\begin{reminder}[Notation used in Proposition~\ref{prop:6-6}]%
  We refer the reader to Definitions~\ref{def:2-22} and \ref{def:2-24} for the
  precise meaning of the conclusion that ``differentials in $V$ annihilate the
  foliation defined by $η_V$.''
\end{reminder}

We prove Proposition~\ref{prop:6-6} in Section~\vref{sec:6-3}.  Before starting
the proof, we draw a number of corollaries that will be instrumental when we
establish Theorem~\ref{thm:1-3} in Section~\ref{sec:7} below.

\begin{cor}[Linear systems defined by strict wedge subspaces]
  In the setup of Proposition~\ref{prop:6-6}, the image of the rational map
  $η_V$ has dimension $\dim \img η_V = k$.
\end{cor}
\begin{proof}
  If
  \[
    𝒜 := \ker \Bigl( Ω¹_X(\log D) → \sHom_{𝒪_X} \bigl( \ker η_V,\, 𝒪_X(D)\bigr) \Bigr) ⊆ Ω¹_X(\log D)
  \]
  is the annihilator of the foliation defined by $η_V$ and if $\sV ⊆ Ω¹_X(\log
  D)$ is the sheaf of differentials introduced in Notation~\ref{not:6-3}, then
  Proposition~\ref{prop:6-6} asserts that $\sV ⊆ 𝒜$.  In particular, we find
  that
  \[
    k = \rank \sV ≤ \rank 𝒜 = \dim \img η_V.
  \]
  The converse is shown in Remark~\ref{rem:3-4}.
\end{proof}

\begin{cor}[Linear systems defined by strict wedge subspaces]\label{cor:6-9}%
  Assume the setup of Proposition~\ref{prop:6-6}.  Then, there exists a dense
  open subset $X° ⊆ X$ where the rational map $η_V$ is well-defined and $\sV =
  \img d η_V$.  \qed
\end{cor}

\begin{cor}[Sums of sheaves of differentials defined by strict wedge subspaces]
  Let $X$ be a compact Kähler manifold, let $D ∈ \Div(X)$ be a reduced snc
  divisor, and let $V_1, …, V_a ⊆ H⁰\bigl( X,\, Ω¹_X(\log D) \bigr)$ be strict
  wedge subspaces, with associated sheaves $\sV_1, …, \sV_a ⊆ Ω¹_X(\log D)$ of
  differentials.  Write $\sV := \sum_i \sV_i ⊆ Ω¹_X(\log D)$ and consider the
  determinant $\det \sV ⊆ Ω^{\rank \sV}_X(\log D)$.  Then,
  \[
    \rank \sV = \dim \img φ_{\det \sV}.
  \]
  In particular, $\det \sV$ is a Bogomolov sheaf in the sense of
  Remark~\ref{rem:3-4}.
\end{cor}
\begin{proof}
  Set $V := \sum_i V_i ⊆ H⁰\bigl( X,\, Ω¹_X(\log D) \bigr)$ and consider the
  rational map
  \[
    \begin{tikzcd}[column sep=2.2cm]
      X \ar[r, dashed, "η \::=\: η_{V_1} ⨯ ⋯ ⨯ η_{V_a}"] & ℙ \Bigl( \img λ_{V_1} \Bigr) ⨯ ⋯ ⨯ ℙ \Bigl( \img λ_{V_a} \Bigr).
  \end{tikzcd}
  \]
  Combining results obtained so far, there exists an open subset $X° ⊆ X$ where
  all maps encountered so far are well-defined and the following chain of
  (in)equalities holds.
  \begin{align*}
    \rank \sV = \textstyle \rank \sum_i \sV_i|_{X°} & = \textstyle \rank \sum_i \img (d η_{V_i}|_{X°}) && \text{Corollary~\ref{cor:6-9}} \\
    & = \rank \img (d η|_{X°}) && \text{Lemma~\ref{lem:2-7}} \\
    & = \dim \img η \\
    & ≤ \dim \img η_V && \text{Corollary~\ref{cor:3-8}} \\
    & ≤ \dim \img φ_{\det \sV} && \text{Remark~\ref{rem:3-4}}
  \end{align*}
  The converse inequality, $\dim \img φ_{\det \sV} ≤ \rank \sV$, holds by
  Remark~\ref{rem:3-4}.
\end{proof}

\begin{cor}[Invariant Bogomolov sheaves defined by strict wedge subspaces]\label{cor:6-11}%
  Let $X$ be a compact Kähler manifold, equipped with a holomorphic action of a
  (possibly infinite) group $G$.  Let $D ∈ \Div(X)$ be a reduced snc divisor
  that is $G$-stable.  Assume that $X$ admits a strict wedge subspace $V ⊆
  H⁰\bigl( X,\, Ω¹_X(\log D) \bigr)$ with associated sheaf $\sV$ of
  differentials, and write $𝒲 := \sum_{g ∈ G} g^* \sV$.  Then,
  \[
    \rank 𝒲 = \dim \img φ_{\det 𝒲}.
  \]
  In particular, $\det \bigl( 𝒲 \bigr) ⊆ Ω^{\rank 𝒲}_X(\log D)$ is a
  $G$-invariant Bogomolov sheaf.
\end{cor}
\begin{proof}
  The Noetherian condition guarantees that $\sum_{g ∈ G} g^* \sV$ is in fact a
  finite sum.
\end{proof}

\subsection{Proof of Proposition~\ref*{prop:6-6}}
\label{sec:6-3}
\approvals{Erwan & yes \\ Frédéric & yes \\ Stefan & yes}

For the reader's convenience, we subdivide the proof into steps.

\subsubsection*{Step 0: Simplification and Notation}
\approvals{Erwan & yes \\ Frédéric & yes \\ Stefan & yes}

The definition of ``strict $k$-wedge subspace'' guarantees that every linear
subspace $V' ⊆ V$ is again a strict $k$-wedge subspace, as long as $\dim V' >
k$.  The factorization found in Proposition~\ref{prop:3-5},
\[
  \begin{tikzcd}[column sep=2cm]
    X \ar[r, dashed, "η_V"'] \ar[rr, bend left=15, dashed, "η_{V'}"] & ℙ \Bigl( \img λ_V \Bigr) \ar[r, dashed, two heads, "∃ η_{V,V'}"'] & ℙ \Bigl( \img λ_{V'} \Bigr),
  \end{tikzcd}
\]
equips us with an inclusion of the foliations defined by $η_V$ and $η_{V'}$,
\[
  \ker(η_V) ⊆ \ker(η_{V'}),
\]
which allows us to assume the following.

\begin{asswlog}
  The dimension of $V$ equals $k+1$.
\end{asswlog}

Resolving the indeterminacies of $η_V$ by a suitable blow-up, $π : \wtilde{X} →
X$ and replacing $V ⊆ H⁰\bigl( X,\, Ω¹_X(\log D) \bigr)$ by its pull-back
$(dπ)(V) ⊆ H⁰\bigl( \wtilde{X},\, Ω¹_{\wtilde{X}}(\log π^*D) \bigr)$, we may
assume the following.

\begin{asswlog}
  The meromorphic mapping $η_V$ is in fact holomorphic.
\end{asswlog}

Denote the image variety by $Y := \img η_V ⊆ ℙ(\img λ_V)$ and let $Y° ⊂ Y$ be
the maximal open subset over which the morphism $η_V$ is a proper submersion.

\subsubsection*{Step 1: Coordinates}
\CounterStep
\approvals{Erwan & yes \\ Frédéric & yes \\ Stefan & yes}

We will work with explicit coordinates and choose an ordered basis $σ_1, …,
σ_{k+1} ∈ V$.  A natural ordered basis of $Λ^k V$ is then given as
\[
  (σ_1 Λ ⋯ Λ \underbrace{σ_j}_{\text{delete}} Λ ⋯ Λ σ_{k+1})_{1 ≤ j ≤ k+1} ∈ Λ^k V.
\]
Use this basis to identify $ℙ(\img λ_V) ≅ ℙ(Λ^k V) ≅ ℙ^k$ and use homogeneous
coordinates $[x_1 : ⋯ : x_{k+1}]$ to denote its points.  This choice of
coordinates allows writing $η_V$ in terms of concrete functions.  To make this
statement precise, observe that the forms $σ_1, …, σ_{k+1}$ span the rank{-}$k$
sheaf $\sV$ generically.  Consequently, there exist meromorphic functions $α_1,
…, α_k ∈ ℳ(X)$ such that
\begin{equation}\label{eq:6-14-1}
  σ_{k+1} = \sum_{j=1}^k α_j·σ_j.
\end{equation}
This implies in particular that
\[
  σ_1 Λ ⋯ Λ \underbrace{σ_j}_{\text{delete}} Λ ⋯ Λ σ_{k+1} = (-1)^{k-j}·α_j·σ_1 Λ ⋯ Λ σ_k,
  \quad \text{for every index } j \le k.
\]
On the open set $U ⊆ X$ where the $α_1, …, α_k$ are regular, the morphism $η_V$
is thus described as
\[
  η_V|_U : U → ℙ^k, \quad x ↦ \bigl[(-1)^{k-1}·α_1(x) : ⋯ : (-1)^{k-k}·α_k(x) : 1\bigr].
\]

\subsubsection*{Step 2: Forms on Fibres}
\approvals{Erwan & yes \\ Frédéric & yes \\ Stefan & yes}

Given any point $s ∈ Y° ∖ \{ x_{k+1} = 0\}$, write $s = [y_1 : ⋯ : y_k : 1]$ and
consider the holomorphic differential form
\[
  τ_s := σ_{k+1} - \sum_{i=1}^k (-1)^{k-i}·y_i·σ_i.
\]
Linear independence of $σ_1, …, σ_{k+1}$ guarantees that $τ_s$ does not vanish
identically on $X$.  Equation~\eqref{eq:6-14-1} however guarantees that $τ_s$
vanishes identically along the fibre $X_s := η_V^{-1}(s)$.
Proposition~\ref{prop:2-21} then guarantees that $τ_s$ annihilates the foliation
defined by $η_V$ over the open set $η^{-1}_V(Y°)$ and hence everywhere.

\subsubsection*{Step 3: Non-degeneracy of the image}
\approvals{Erwan & yes \\ Frédéric & yes \\ Stefan & yes}

The construction of $η_V$ guarantees that the image variety $Y$ is \emph{not}
linearly degenerate inside $ℙ^k$.  In other words, $Y$ is not contained in any
linear hyperplane.  If $(s_1, …, s_{k}) ∈ \bigl(Y° ∖ \{ x_{k+1} = 0\}\bigr)^{⨯
k}$ is a general $k$-uple of points, written as $s_i = [y_{i1} : ⋯ : y_{ik} : 1]$,
then non-degeneracy implies that
\[
  (y_{11}, …, y_{1k}, 1), …, (y_{(k)1}, …, y_{(k)k}, 1) ∈ ℂ^{k+1}
\]
are linearly independent vectors.  The holomorphic forms
\[
  τ_{s_i} := σ_{k+1} - \sum_{j=1}^k (-1)^{k-i}·y_{ij}·σ_j
\]
are likewise linearly independent, hence form a basis of $V$, and annihilate the
foliation defined by $η_V$.  \qed

%
%
\svnid{$Id: 07-main.tex 130 2026-01-07 16:08:21Z kebekus $}
\selectlanguage{british}

\section{Bounds on invariants, Proof of Theorem~\ref*{thm:1-3} and Corollary~\ref*{cor:1-7}}
\subversionInfo
\label{sec:7}

\subsection{Proof of Theorem~\ref*{thm:1-3}}
\label{sec:7-1}
\approvals{Erwan & yes \\ Frédéric & yes \\ Stefan & yes}

We prove the contrapositive: assuming that $(X, D)$ is an $n$-dimensional
$\cC$-pair with irregularity $q⁺(X,D) > n$ satisfying the assumptions of
Theorem~\ref*{thm:1-3}, we will show that $(X, D)$ is \emph{not} special.

\subsection*{Step 1: Setup}
\approvals{Erwan & yes \\ Frédéric & yes \\ Stefan & yes}
\CounterStep

The assumption that $q⁺(X,D) > n$ allows choosing a Galois cover $γ : \what{X} →
X$ with irregularity $q(X,D,γ) > n$.  Fix one such cover throughout and denote
its Galois group by $G$.  Consider the reduced divisor $\what{D} := (γ^*
⌊D⌋)_{\red}$ and let $π : \wtilde{X} → \what{X}$ be a $G$-equivariant, strict
log resolution of the pair $\bigl(\what{X}, \what{D} \bigr)$,
\begin{equation}\label{eq:7-1-1}
  \begin{tikzcd}[column sep=2.2cm]
    \wtilde{X} \ar[r, two heads, "π\text{, resolution}"] & \what{X} \ar[r, two heads, "γ\text{, cover}"] & X.
  \end{tikzcd}
\end{equation}
The preimage $π^{-1} \bigl(\supp \what{D} \bigr)$ is then $G$-invariant, of
pure codimension one and has simple normal crossing support.  Let $\wtilde{D} ∈
\Div\bigl(\wtilde{X}\bigr)$ be the associated divisor.  The pair $(\wtilde{X},
\wtilde{D})$ is snc, and the sheaf $Ω¹_{\wtilde{X}}(\log \wtilde{D})$ is a $G$-sheaf.

\begin{claim}\label{claim:7-2}
  There exists an injective sheaf morphism
  \[
    d_{\cC} π : π^{[*]} Ω^{[1]}_{(X,D,γ)} ↪ Ω¹_{\wtilde{X}}(\log \wtilde D)
  \]
  that agrees on the Zariski open set $π^{-1} (\what{X}°)$ with the standard
  pull-back of Kähler differentials, in the sense discussed in
  Explanation~\ref{exp:4-8}.
\end{claim}
\begin{proof}[Proof of Claim~\ref{claim:7-2}]
  We consider the alternative assumptions of Theorem~\ref{thm:1-3} separately.
  In case~\ref{il:1-3-1}, where $(X,D)$ is locally uniformizable, this is
  \cite[Fact~5.9]{orbiAlb1}.  In case~\ref{il:1-3-2}, where $X$ is projective and
  $(X,D)$ is dlt, this is Theorem~\ref{thm:1-6}, in the formulation given by
  Proposition~\ref{prop:4-7}.  \qedhere~\mbox{(Claim~\ref{claim:7-2})}
\end{proof}

By minor abuse of notation, we suppress $d_{\cC} π$ from the notation and view
$π^{[*]} Ω^{[1]}_{(X,D,γ)}$ as a $G$-subsheaf of the $G$-sheaf
$Ω¹_{\wtilde{X}}(\log \wtilde D)$.

\subsection*{Step 2: An invariant Bogomolov sheaf on $\wtilde{X}$}
\approvals{Erwan & yes \\ Frédéric & yes \\ Stefan & yes}
\CounterStep

We know by assumption that
\[
  h⁰ \bigl( \wtilde{X},\, π^{[*]} Ω^{[1]}_{(X,D,γ)} \bigr)
  = h⁰ \bigl( \what{X},\, Ω^{[1]}_{(X,D,γ)} \bigr) = q(X,D,γ) > n.
\]
Use Proposition~\ref{prop:6-5} to find a number $p$ and a strict $p$-wedge
subspace
\[
  V ⊆ H⁰\bigl( \wtilde{X},\, π^{[*]} Ω^{[1]}_{(X,D,γ)} \bigr) ⊆ H⁰\bigl( \wtilde{X},\, Ω¹_{\wtilde{X}}(\log \wtilde{D}) \bigr)
\]
with the associated sheaf $\sV ⊆ π^{[*]} Ω^{[1]}_{(X,D,γ)} ⊆
Ω¹_{\wtilde{X}}(\log \wtilde{D})$ of differentials.  Denoting the sum of the
pull-back sheaves by
\[
  𝒲 := \sum_{g ∈ G} g^* \sV ⊆ π^{[*]} Ω^{[1]}_{(X,D,γ)} ⊆ Ω¹_{\wtilde{X}}(\log \wtilde{D}),
\]
Corollary~\ref{cor:6-11} shows that
\begin{equation}\label{eq:7-3-1}
  \det \bigl( 𝒲 \bigr) ⊆ Λ^{[\rank 𝒲]} π^{[*]} Ω^{[1]}_{(X,D,γ)} ⊆ Ω^{\rank 𝒲}_{\wtilde{X}}(\log \wtilde{D})
\end{equation}
satisfies $\rank 𝒲 ≤ \dim \img φ_{\det 𝒲}$ and is hence a $G$-invariant
Bogomolov sheaf.

\subsection*{Step 3: An invariant Bogomolov sheaf on $\what{X}$}
\approvals{Erwan & yes \\ Frédéric & yes \\ Stefan & yes}

We conclude the proof of Theorem~\ref{thm:1-3} by exhibiting the push-forward
$π_* \det 𝒲$ as a $G$-invariant Bogomolov sheaf on $\what{X}$.  The next claim
allows viewing the push-forwards as a sheaf of adapted reflexive differentials.

\begin{claim}\label{claim:7-4}
  There exists a natural inclusion of $G$-sheaves,
  \[
    π_* Λ^{[\rank 𝒲]} π^{[*]} Ω^{[1]}_{(X,D,γ)} ↪ Ω^{[\rank 𝒲]}_{(X,D,γ)}.
  \]
\end{claim}
\begin{proof}[Proof of Claim~\ref{claim:7-4}]
  The sheaf on the left is the push-forward of a torsion free sheaf, and hence
  itself torsion free.  It agrees with the sheaf on the right over the big open
  set $(\what{X},\what{D})_{\reg} ⊆ \what{X}$ where the strict resolution map
  $π$ is isomorphic.  In other words: denoting the inclusion by $ι :
  (\what{X},\what{D})_{\reg} ↪ \what{X}$, we find an isomorphism
  \[
    \begin{tikzcd}[column sep=1.6cm]
      ι^* π_* Λ^{[\rank 𝒲]} π^{[*]} Ω^{[1]}_{(X,D,γ)} \ar[r, hook, two heads, "\text{isomorphic}"] & ι^* Ω^{[\rank 𝒲]}_{(X,D,γ)}.
    \end{tikzcd}
  \]
  Pushing forward, this gives a sequence of sheaf morphisms,
  \[
    \begin{tikzcd}
      π_* Λ^{[\rank 𝒲]} π^{[*]} Ω^{[1]}_{(X,D,γ)} \ar[r, hook] & ι_* ι^* π_* Λ^{[\rank 𝒲]} π^{[*]} Ω^{[1]}_{(X,D,γ)} \ar[r, hook, two heads] & ι_* ι^* Ω^{[\rank 𝒲]}_{(X,D,γ)},
    \end{tikzcd}
  \]
  where injectivity of the first arrow comes from the fact that $π_* Λ^{[\rank
  𝒲]} π^{[*]} Ω^{[1]}_{(X,D,γ)}$ is torsion free.  To conclude, it suffices to
  note that $Ω^{[\rank 𝒲]}_{(X,D,γ)}$ is reflexive, and hence equal to the last
  term in the sequence, $ι_* ι^* Ω^{[\rank 𝒲]}_{(X,D,γ)}$.
  \qedhere~\mbox{(Claim~\ref{claim:7-4})}
\end{proof}

Inclusion~\eqref{eq:7-3-1} and Claim~\ref{claim:7-4} equip us with an inclusion
of $G$-sheaves,
\[
  π_* \det 𝒲 ⊆ Ω^{[\rank 𝒲]}_{(X,D,γ)}.
\]
The natural isomorphism $H⁰ \bigl( \wtilde{X},\, \det 𝒲 \bigr) = H⁰ \bigl(
\what{X},\, π_* \det 𝒲 \bigr)$ will then show that
\[
  \rank 𝒲 ≤ \dim \img φ_{\det 𝒲} = \dim \img φ_{π_* \det 𝒲}.
\]
The sheaf $π_* \det 𝒲$ is thus a $G$-invariant Bogomolov sheaf, and the claim
follows from Corollary~\ref{cor:5-4}.  This finishes the proof of
Theorem~\ref{thm:1-3}.  \qed

\subsection{Proof of Corollary~\ref*{cor:1-7}}
\approvals{Erwan & yes \\ Frédéric & yes \\ Stefan & yes}

Let $(X, D)$ be a projective $\cC$-pair that is dlt.  Assuming the existence of
a cover $γ : \what{X} \twoheadrightarrow X$ and a rank-one sheaf $ℒ ⊆
Ω^{[1]}_{(X,D,γ)}$ of positive $\cC$-Kodaira-Iitaka dimension, $κ_{\cC}(ℒ) > 0$,
we need to show that $(X,D)$ is \emph{not} special.  The proof is largely
parallel to the proof of Theorem~\ref{thm:1-3} given in the previous
Section~\ref{sec:7-1}.

\subsection*{Step 0: Simplifying assumptions}
\approvals{Erwan & yes \\ Frédéric & yes \\ Stefan & yes}
\CounterStep

Given a sequence of covers,
\begin{equation}\label{eq:7-5-1}
  \begin{tikzcd}[column sep=2.4cm]
    \wcheck{X} \ar[r, two heads, "β\text{, Galois cover}"'] \ar[rr, bend left=15, two heads, "α\text{, cover}"] & \what{X} \ar[r, two heads, "γ\text{, cover}"'] & X,
  \end{tikzcd}
\end{equation}
Proposition~\ref{prop:5-1} equips us with rank-one, reflexive subsheaf
$\wcheck{ℒ} ⊆ Ω^{[1]}_{(X,D,γ◦β)}$ on $\wcheck{X}$ with positive
$\cC$-Kodaira-Iitaka dimension, $κ_{\cC}(\wcheck{ℒ}) ≥ κ_{\cC}(ℒ) > 0$.  We will
use this observation to replace the cover $γ$ by higher covers with additional
properties.

\subsection*{Step 0.1: The cover is adapted}
\approvals{Erwan & yes \\ Frédéric & yes \\ Stefan & yes}

The projectivity assumption and \cite[Lem.~2.36]{orbiAlb1} imply that $X$ admits
an adapted cover.  An elementary fibre product construction will then allow us
to produce a Sequence~\eqref{eq:7-5-1} of covers where $α$ is itself adapted.

\begin{asswlog}\label{asswlog:7-6}%
  The cover $γ$ is adapted and Galois, with Group $G$.
\end{asswlog}

\subsection*{Step 0.2: The sheaf $ℒ$ admits sections}
\approvals{Erwan & yes \\ Frédéric & yes \\ Stefan & yes}
\CounterStep

The assumption on the positivity of the $\cC$-Kodaira-Iitaka dimension,
$κ_{\cC}(ℒ) > 0$, equips us with number $m > 0$ and two linearly independent
sections
\begin{equation}\label{eq:7-7-1}
  σ_0, σ_1 ∈ H⁰ \bigl( \what{X},\, \Sym^{[m]}_{\cC} ℒ \bigr).
\end{equation}
The right side of \eqref{eq:7-7-1} simplifies because of
Assumption~\ref{asswlog:7-6}.  To be precise, recall from
\cite[Obs.~4.12]{orbiAlb1} that
\[
  \Sym^{[m]}_{\cC} ℒ = ℒ^{[⊗ m]}
  \quad \text{so that} \quad
  σ_0, σ_1 ∈ H⁰ \bigl( \what{X},\, ℒ^{[⊗ m]} \bigr).
\]
Standard covering constructions (``taking the $m$-th root out of a section'')
produce a Sequence~\eqref{eq:7-5-1} of covers and sections
\[
  τ_0, τ_1 ∈ H⁰ \bigl( \wcheck{X},\, β^{[*]} ℒ \bigr)
\]
whose $m$-th tensor powers agree over a dense open set with the pullbacks of
$σ_0$ and $σ_1$.  Recalling from \cite[Obs.~4.8]{orbiAlb1} that the sheaf
$β^{[*]} ℒ$ injects into $Ω^{[1]}_{(X,D,γ◦β)}$, we may replace $γ$ by $α$ and
make the following assumption.

\begin{asswlog}\label{asswlog:7-8}%
  The sheaf $ℒ$ admits two linearly independent sections.
\end{asswlog}

\subsection*{Step 1: Setup}
\approvals{Erwan & yes \\ Frédéric & yes \\ Stefan & yes}

Consider the reduced divisor $\what{D} := (γ^* ⌊D⌋)_{\red}$ and let $π :
\wtilde{X} → \what{X}$ be a $G$-equivariant, strict log resolution of the pair
$\bigl(\what{X}, \what{D} \bigr)$, as in \eqref{eq:7-1-1} above.  Taking
$\wtilde{D} ∈ \Div\bigl(\wtilde{X}\bigr)$ as the divisor supported on $π^{-1}
\bigl(\supp \what{D} \bigr)$, we find the following.

\begin{claim}[Analogue of Claim~\ref{claim:7-2}]\label{claim:7-9}%
  There exists an injective sheaf morphism
  \[
    d_{\cC} π : π^{[*]} Ω^{[1]}_{(X,D,γ)} ↪ Ω¹_{\wtilde{X}}(\log \wtilde D)
  \]
  that agrees on the Zariski open set $π^{-1} (\what{X}°)$ with the standard
  pull-back of Kähler differentials, in the sense discussed in
  Explanation~\ref{exp:4-8}.  \qed
\end{claim}

Suppress $d_{\cC} π$ from the notation and view $π^{[*]} Ω^{[1]}_{(X,D,γ)}$ as a
$G$-subsheaf of the $G$-sheaf $Ω¹_{\wtilde{X}}(\log \wtilde D)$.

\subsection*{Step 2: An invariant Bogomolov sheaf on $\wtilde{X}$}
\approvals{Erwan & yes \\ Frédéric & yes \\ Stefan & yes}

Use Assumption~\ref{asswlog:7-8} to define a strict $1$-wedge subspace
\[
  V := \langle τ_0, τ_1 \rangle
  ⊆ H⁰ \bigl( \what{X},\, Ω^{[1]}_{(X,D,γ)} \bigr)
  ⊆ H⁰\bigl( \wtilde{X},\, π^{[*]} Ω^{[1]}_{(X,D,γ)} \bigr)
  \overset{\text{Claim~\ref{claim:7-9}}}{⊆} H⁰\bigl( \wtilde{X},\, Ω¹_{\wtilde{X}}(\log \wtilde{D}) \bigr)
\]
with the associated sheaf $\sV ⊆ π^{[*]} Ω^{[1]}_{(X,D,γ)} ⊆
Ω¹_{\wtilde{X}}(\log \wtilde{D})$ of differentials.  Denoting the sum of the
pull-back sheaves by
\[
  𝒲 := \sum_{g ∈ G} g^* \sV ⊆ π^{[*]} Ω^{[1]}_{(X,D,γ)} ⊆ Ω¹_{\wtilde{X}}(\log \wtilde{D}),
\]
Corollary~\ref{cor:6-11} shows that $\det \bigl( 𝒲 \bigr) ⊆ Ω^{\rank
𝒲}_{\wtilde{X}}(\log \wtilde{D})$ satisfies $\rank 𝒲 ≤ \dim \img φ_{\det 𝒲}$
and is hence a $G$-invariant Bogomolov sheaf.

\subsection*{Step 3: An invariant Bogomolov sheaf on $\what{X}$}
\approvals{Erwan & yes \\ Frédéric & yes \\ Stefan & yes}

In analogy with the proof of Theorem~\ref{thm:1-3}, observe that there exists a
natural inclusion of $G$-sheaves, $π_* \det 𝒲 ⊆ Ω^{[\rank 𝒲]}_{(X,D,γ)}$, that
exhibits $π_* \det 𝒲$ is as a $G$-Bogomolov sheaf.  This finishes the proof of
Corollary~\ref{cor:1-7}.  \qed


\appendix
\part*{Appendix}

%
%
\svnid{$Id: 0A-smallextension.tex 134 2026-01-09 08:49:53Z kebekus $}
\selectlanguage{british}

\section{Extension of low degree differentials}
\subversionInfo
\label{sec:A}
\approvals{Erwan & yes \\ Frédéric & yes \\ Stefan & yes}

Let $X$ be a normal analytic variety and let $π : \wtilde{X} → X$ be a
resolution of singularities.  If the singular set of $X$ is small, $\codim_X
X_{\sing} ≥ p + 2$ for one $p ∈ ℕ$, then Flenner has shown in \cite{Flenner88}
that $p$-forms extend from the smooth locus of $X$ to $p$-forms on $\wtilde{X}$:
The natural restriction map
\[
  H⁰ \bigl ( \wtilde{X},\, Ω^p_{\wtilde{X}} \bigr) ↪ H⁰ \bigl ( π^{-1}(X_{\reg}),\, Ω^p_{\wtilde{X}} \bigr) = H⁰ \bigl( X_{\reg},\, Ω^p_X \bigr)
\]
is isomorphic.  Writing $ι : X_{\reg} ↪ X$ for the inclusion map, Flenner's
result can equivalently be stated by saying the natural inclusion
\[
  π_* Ω^p_{\wtilde{X}} → ι_* ι^* π_* Ω^p_{\wtilde{X}}
\]
is isomorphic.  The proof builds on earlier work \cite{SS85} of Steenbrink and
van~Straten.  It relies on relative duality for cohomology with supports and
Hodge-theoretic methods.

\subsection{Main Result}
\approvals{Erwan & yes \\ Frédéric & yes \\ Stefan & yes}

The present section extends Flenner's result to forms with logarithmic poles and
replaces $X_{\sing}$ with arbitrary subvarieties $Z ⊆ X$: If $Z$ is small,
$\codim_X Z ≥ p + 2$ for one $p ∈ ℕ$, then $p$-forms extend from $\wtilde{X} ∖
π^{-1}(Z)$ to $p$-forms on $\wtilde{X}$.

\begin{thm}[Extension of low degree differentials]\label{thm:A-1}%
  Let $(X, D)$ be a logarithmic pair, let $π : \wtilde{X} → X$ be a log
  resolution of singularities, and let $\wtilde{D} ∈ \Div(\wtilde{X})$ be the
  reduced divisor supported on $π^{-1} D$.  Let $p ∈ ℕ$ be any number.  If $Z ⊂
  X$ is a Zariski closed subset of $\codim_X Z ≥ p+2$, then the natural
  restriction map
  \begin{equation}\label{eq:A-1-1}
    H⁰ \left( \wtilde{X},\, Ω^p_{\wtilde{X}} \bigl( \log \wtilde{D} \bigr) \right)
    ↪ H⁰ \left( \wtilde{X} ∖ π^{-1}(Z),\, Ω^p_{\wtilde{X}}\bigl(\log \wtilde{D}\bigr) \right)
  \end{equation}
  is isomorphic.
\end{thm}

\begin{cor}[Extension of low degree differentials]\label{cor:A-2}%
  In the setting of Theorem~\ref{thm:A-1}, write $ι : X ∖ Z → X$ for the
  inclusion map.  Then, natural morphism
  \[
    π_* Ω^p_{\wtilde{X}}(\log \wtilde{D}) → ι_* ι^* π_* Ω^p_{\wtilde{X}}(\log \wtilde{D})
  \]
  is isomorphic.  \qed
\end{cor}

Theorem~\ref{thm:A-1} will be shown in Section~\ref{sec:A-4} below.

\subsection{Idea of application}
\approvals{Erwan & yes \\ Frédéric & yes \\ Stefan & yes}

Theorem~\ref{thm:A-1} will typically be used in scenarios where the singular
locus $X_{\sing}$ comes with a stratification and $Z$ is a (potentially strict)
subvariety of $X_{\sing}$.  One can then consider restrictions,
\begin{align*}
  H⁰ \left( \wtilde{X},\, Ω^p_{\wtilde{X}} \bigl( \log \wtilde{D} \bigr) \right) & \overset{α}{↪} H⁰ \left( \wtilde{X} ∖ π^{-1}(Z),\, Ω^p_{\wtilde{X}}\bigl(\log \wtilde{D}\bigr) \right) \\
  & \overset{β}{↪} H⁰ \left( π^{-1}(X_{\reg}),\, Ω^p_{\wtilde{X}}\bigl(\log \wtilde{D}\bigr) \right) = H⁰ \left( X_{\reg},\, Ω^p_X\bigl(\log D\bigr) \right)
\end{align*}
and ask if a given form $σ° ∈ H⁰ \left( X_{\reg},\, Ω^p_X\bigl(\log D\bigr)
\right)$ is induced by a form on $\wtilde{X}$.  If one knows that $σ°$ lies in
the image of $β$ (for instance, because the singularities of $X$ are
sufficiently mild outside the stratum $Z$), then Theorem~\ref{thm:A-1} might
apply to show that $σ°$ already lies in the image of $β◦α$.

These ideas are put to work in Section~\ref{sec:4-3}, where we prove
Theorem~\ref{thm:1-6}.  There, $Z ⊂ X_{\sing}$ is a stratum of high
codimension, such that $X$ has no worse than quotient singularities outside $Z$.

\subsection{Relation to the literature}
\approvals{Erwan & yes \\ Frédéric & yes \\ Stefan & yes}

To the best of our knowledge, Theorem~\ref{thm:A-1} has not appeared in the
literature before.  The arguments used in the proof are however not new, and
likely known among experts.  Our proof combines ideas of Graf, who replaces
duality for cohomology with supports by numerical computation involving Kodaira
dimension and Bogomolov-Sommese vanishing, with work of Núñez, who follows
\cite{KS18} by expressing extension properties in terms of the dimension of the
support of certain local cohomology sheaves.

Despite appearance to the contrary, the results presented here are unrelated to
the general extension results of \cite{GKKP11, KS18} that are much more
difficult to prove.

\subsection{Proof of Theorem~\ref*{thm:A-1}}
\label{sec:A-4}
\approvals{Erwan & yes \\ Frédéric & yes \\ Stefan & yes}

For completeness' sake, we give a full and mostly self-contained proof,
combining results and ideas from \cite{MR4280499, Nunez24}.  The reader will
note that Steps~3--5 are copied with only minor modifications from
\cite[p.~599f]{MR4280499}.

\subsection*{Step 1: Setup}
\approvals{Erwan & yes \\ Frédéric & yes \\ Stefan & yes}

The restriction map \eqref{eq:A-1-1} is clearly injective.  To show
surjectivity, let
\[
  σ° ∈ H⁰ \left( \wtilde{X} ∖ π^{-1}(Z),\, Ω^p_{\wtilde{X}}\bigl(\log \wtilde{D}\bigr) \right)
\]
be any form.  We need to find a differential form $σ ∈ H⁰ \left( \wtilde{X},\,
Ω^p_{\wtilde{X}}\bigl(\log \wtilde{D}\bigr) \right)$ whose restriction to
$\wtilde{X} ∖ π^{-1}(Z)$ agrees with $σ°$.

\subsection*{Step 2: Simplification}
\approvals{Erwan & yes \\ Frédéric & yes \\ Stefan & yes}

The case where $D ∈ \Div(X)$ is the zero divisor has been worked out by Núñez in
\cite[Lem.~24]{Nunez24}.  Applying Núñez' result to $X ∖ \supp D$, we already
obtain a form on $\wtilde{X} ∖ π^{-1}(Z ∩ D)$ whose restriction to $\wtilde{X} ∖
π^{-1}(Z)$ agrees with $σ°$.  This allows making the following assumption.

\begin{asswlog}
  The set $Z$ is contained in the support of $D$.
\end{asswlog}

Since any two log resolutions of singularities are dominated by a common third,
it is clear that validity of Theorem~\ref{thm:A-1} does not depend on the choice
of the log resolution $π$.  Replacing $π$ by a different resolution if need be,
we may assume the following.

\begin{asswlog}
  The preimage $π^{-1}(Z)$ is of pure codimension one in $\wtilde{X}$ and has
  simple normal crossing support.
\end{asswlog}

Write $E ∈ \Div(\wtilde{X})$ for the reduced divisor supported on $π^{-1}(Z)$.

\subsection*{Step 3: Extension as a rational form}
\approvals{Erwan & yes \\ Frédéric & yes \\ Stefan & yes}

Grauert’s Direct Image Theorem, \cite[Chapt.~10.4]{CAS}, implies that the sheaf
$π_* Ω¹_{\wtilde{X}}\bigl(\log \wtilde{D}\bigr)$ is coherent.  The pull-back
$σ°$ therefore extends as a meromorphic log-form $τ$ on $\wtilde{X}$.
Quantifying the poles, let $G ∈ \Div\bigl(\wtilde{X}\bigr)$ be the minimal
effective divisor such that there exists a section
\[
  τ ∈ H⁰ \left( \wtilde{X},\, Ω^p_{\wtilde{X}}\bigl(\log \wtilde{D}\bigr)(G) \right)
  = \Hom \left( 𝒪_{\wtilde{X}}(-G),\, Ω^p_{\wtilde{X}}\bigl(\log \wtilde{D}\bigr) \right)
\]
whose restriction to $\wtilde{X} ∖ π^{-1}(Z)$ agrees with $σ°$.  Aiming to show
that $G = 0$, we argue by contradiction.

\begin{assumption}\label{ass:A-5}%
  The divisor $G$ is not zero.
\end{assumption}

Minimality of $G$ implies that $G$ is $π$-exceptional.  Since $π$ is a log
resolution, $G$ will then have simple normal crossing support.
Assumption~\ref{ass:A-5} and \cite[Prop.~4]{MR4280499} imply the existence of an
irreducible component $P ⊆ \supp G$ such that the invertible sheaf
$𝒪_{\wtilde{X}}(-G)$ is big when restricted a general fibre of $π|_P$.
Minimality of $G$ guarantees that the restricted map of sheaves on $P$,
\[
  τ|_P : \left.𝒪_{\wtilde{X}}(-G)\right|_P → \left.Ω^p_{\wtilde{X}}\bigl(\log \wtilde{D}\bigr)\right|_P,
\]
does not vanish.

\subsection*{Step 4: Residue sequence}
\approvals{Erwan & yes \\ Frédéric & yes \\ Stefan & yes}

Considering the divisor $P^c := (E - P)|_P ∈ \Div(P)$, the residue sequence for
$p$-forms along $P$ reads
\[
  \begin{tikzcd}
    & & \left.𝒪_{\wtilde{X}}(-G)\right|_P \ar[d, "τ|_P"] \\
    0 \ar[r] & Ω^p_P(\log P^c) \ar[r] & \left.Ω_{\wtilde{X}}^p(\log E)\right|_P \ar[r, "\res_P"'] & Ω_P^{p-1}(\log P^c) \ar[r] & 0.
  \end{tikzcd}
\]
We obtain an injection
\[
  μ : 𝒪_{\wtilde{X}}(-G)|_P ↪ Ω_P^r(\log P^c),
\]
for one $r ∈ \{ p - 1, p \}$.  In both cases, $r ≤ p$.

\subsection*{Step 5: Restriction to the general fibre}
\approvals{Erwan & yes \\ Frédéric & yes \\ Stefan & yes}
\CounterStep

For brevity of notation, write $B := π(P)$ and denote the restricted morphism
$π|_P$ by $ρ : P → B$.  Next, let $F ⊂ P$ be a general fibre of $ρ$ and consider
the restricted divisor $F^c := P^c|_F ∈ \Div(F)$.  Since $(P, P^c)$ is an snc
pair, so is $(F, F^c)$.  The inequality $\codim_X Z ≥ p + 2$ implies that
\begin{align}
  \dim F & = \dim P - \dim B \nonumber\\
  & ≥ (\dim X - 1) - \dim Z \label{eq:A-6-1} \\
  & ≥ (\dim X - 1) - (\dim X - (p + 2)) = p + 1.  \nonumber
\end{align}
Since $F$ is general and $ρ(F)$ is a smooth point of $B$, the restriction of the
standard sequence of relative log differentials, \cite[Sec.~4.1]{EV92}, reads
\[
  0 → \left.ρ^*Ω_B¹\right|_F → \left.Ω_P¹(\log P^c)\right|_F → \left.Ω_{P/B}¹(\log P^c)\right|_F → 0.
\]
By \cite[Ch.~II, Ex.~5.16]{Ha77}, it induces a filtration of $Ω_P^r(\log
P^c)|_F$ with quotients
\[
  ρ^*Ω_Bⁱ ⊗ Ω_{P/B}^{r-i}(\log P^c), \quad i = 0, …, r.
\]
Restricting the morphism $μ$ to the general fibre $F$, we obtain an injection
\[
  μ|_F : 𝒪_{\wtilde{X}}(-G)|_F ↪ ρ^*Ω_Bⁱ ⊗ Ω_{P/B}^{r-i}(\log P^c) = Ω_F^{r-i}(\log F^c)^{⊕ \rank Ω_Bⁱ},
\]
for one suitable number $i$.  Projecting onto a suitable summand, we even find
an inclusion
\[
  𝒪_{\wtilde{X}}(-G)|_F ↪ Ω_F^{r-i}(\log F^c).
\]
This leads to a contradiction.
\begin{itemize}
  \item On the one hand, the classic Bogomolov-Sommese vanishing theorem,
    \cite[Cor.~6.9]{EV92}, implies that $κ\left( 𝒪_{\wtilde{X}}(-G)|_F \right)
    ≤ r - i ≤ r ≤ p$.

  \item On the other hand, the choice of $G$ guarantees that
    $𝒪_{\wtilde{X}}(-G)$ is big when restricted a general fibre of $π|_P$, so
    that $κ\left( 𝒪_{\wtilde{X}}(-G)|_F \right) = \dim F$.  However, we have
    seen in \eqref{eq:A-6-1} that $\dim F ≥ p + 1$.
\end{itemize}
Assumption~\ref{ass:A-5} is therefore absurd.  We conclude that $G = 0$, as
desired.  \qed


\newcommand{\etalchar}[1]{$^{#1}$}

\end{document}